\documentclass{article}

\usepackage{amsmath}
\usepackage{amsthm}
\usepackage[margin=1in]{geometry}
\usepackage[unicode,psdextra]{hyperref}

\usepackage{paper}

\usepackage[numbers]{natbib}
\bibpunct[, ]{[}{]}{,}{n}{}{,}%

\newtheorem{theorem}{Theorem}
\newtheorem{lemma}{Lemma}
\newtheorem{corollary}{Corollary}
\newtheorem{proposition}{Proposition}
\newtheorem{assumption}{Assumption}
\theoremstyle{definition}
\newtheorem{definition}{Definition}
\theoremstyle{remark}
\newtheorem{example}{Example}
\newtheorem{remark}{Remark}
\newtheorem{observation}{Observation}

\def\ackname{Acknowledgments}%
\def\acknowledgment{\par\addvspace{17pt}\small\rmfamily
	\trivlist\if!\ackname!\item[]\else
	\item[\hskip\labelsep
	{\bfseries\ackname}]\fi}

\newenvironment{acknowledgments}{\begin{acknowledgment}}
	{\end{acknowledgment}}

\begin{document}
	\title {Intersection disjunctions for reverse convex sets}
	\author{Eli Towle\thanks{etowle@wisc.edu} \and James Luedtke\thanks{jim.luedtke@wisc.edu}}
	\date{\small Department of Industrial and Systems Engineering, University of Wisconsin -- Madison}
	
	\maketitle
	
	\begin{abstract}
	We present a framework to obtain valid inequalities for a reverse convex set: the set of points in a polyhedron that lie outside a given open convex set. Reverse convex sets arise in many models, including bilevel optimization and polynomial optimization. An intersection cut is a well-known valid inequality for a reverse convex set that is generated from a basic solution that lies within the convex set. We introduce a framework for deriving valid inequalities for the reverse convex set from basic solutions that lie outside the convex set. We first propose an extension to intersection cuts that defines a two-term disjunction for a reverse convex set, which we refer to as an intersection disjunction. Next, we generalize this analysis to a multi-term disjunction by considering the convex set's recession directions. These disjunctions can be used in a cut-generating linear program to obtain valid inequalities for the reverse convex set.
		
		\noindent \rule{0pt}{1.5em}\textbf{Keywords:} Mixed-integer nonlinear programming; valid inequalities; reverse convex sets; disjunctive programming; intersection cuts
	\end{abstract}
	
	\begin{acknowledgments}
		\begin{sloppypar}
			This material is based upon work supported by the U.S.\ Department of Energy, Office of Science, Office of Advanced Scientific Computing Research (ASCR) under Contract DE-AC02-06CH11357. The authors acknowledge partial support through NSF grant SES-1422768.
		\end{sloppypar}
	\end{acknowledgments}

\section{Introduction}\label{sec:intro}

A reverse convex set is a set of the form $\P \setminus \Q$, where $\P \subseteq \R^n$ is a polyhedron and $\Q \subseteq \R^n$ is an open convex set. This is a general set structure arising in the context of mixed-integer nonlinear programming (MINLP). In this setting, $\P$ is a linear programming relaxation of the MINLP feasible region, and $\Q$ contains no solutions feasible to the problem. We are motivated by cases where $\cl(\Q)$ is either non-polyhedral or is defined by a large number of linear inequalities, because if $\cl(\Q)$ is a polyhedron defined by a small number of inequalities, we can optimize and separate over $\clconv(\P \setminus \Q)$ efficiently using the disjunctive programming techniques of \citet{balas1979,balas1974}. We study valid inequalities for reverse convex sets. These inequalities can be used to strengthen the convex relaxation of any problem for which an open convex set containing no feasible points can be identified; such sets are known as {\em convex $S$-free sets} (e.g., \citet{conforti2014b}).

Intersection cuts are valid inequalities for $\P \setminus \Q$. Intersection cuts were introduced in the context of concave minimization by \citet{tuy1964} and later by \citet{balas1971} for integer programming. The inequalities of \citet{tuy1964} are often referred to as ``concavity cuts'' or ``$\gamma$-valid'' cuts. For ease of exposition, we refer to such inequalities as ``intersection cuts'' due to their similarity to the inequalities of \citet{balas1971}. An intersection cut is generated from a basic solution $\xbasis$ of $\P$ that satisfies $\xbasis \in \Q$. A basic solution $\xbasis$ of $\P$ corresponding to basis $B$ forms the apex of a translated simplicial cone $\PB$ defining a relaxation of $\P$. For each extreme ray of this cone, a point on $\bd(\Q)$ that intersects the extreme ray is found. A hyperplane $c^\tp x = d$ is formed, such that the hyperplane passes through all of these points and satisfies $c^\tp \xbasis > d$. The intersection cut $c^\tp x \leq d$ is valid for $\P \setminus \Q$. For a detailed review of intersection cuts, see Section~\ref{sec:ic}.

Our main contribution in this work is to show how valid inequalities for $\P \setminus \Q$ can be obtained from a basic solution $\xbasis$ of $\P$ {\em in the case where $\xbasis \notin \cl(\Q)$}. Because $\xbasis \notin \Q$, intersection cuts generated using the cone $\PB$ are not valid or even well-defined in general. However, under the assumption that each extreme ray of $\PB$ that does not intersect $\Q$ lies in the recession cone of $\Q$, we present two linear inequalities that form a two-term disjunction that contains $\P \setminus \Q$. If $\P$ intersected with one of these inequalities is empty, the inequality defining the other disjunctive term is valid for $\P \setminus \Q$. We call inequalities obtained in this manner {\em external intersection cuts}. If both disjunctive terms are nonempty, we can generate valid inequalities for $\P \setminus \Q$ using the standard cut-generating linear program (CGLP) for disjunctive programming of \citet{balas1979,balas1974}. We refer to these disjunctions as {\em intersection disjunctions}.

\citet{glover1974} exploits the recession structure of $\Q$ to strengthen the intersection cut. This motivates us to extend our analysis by considering $\recc(\Q)$, the recession cone of $\Q$. We present a relaxation of $\PB \setminus \Q$ that incorporates the structure of $\recc(\Q)$. We provide a class of valid inequalities for the relaxation, the size of which grows exponentially with the dimension of the problem. We derive a polynomial-size extended formulation that captures the full strength of this exponential family of inequalities. We then prove that the proposed relaxation of $\PB\setminus \Q$ is equivalent to the union of at most $n$ possibly nonconvex sets, thereby forming a disjunction that contains the reverse convex set. Under some assumptions, we propose a polyhedral relaxation of each disjunctive term individually. Given these polyhedral relaxations, we can use a CGLP to generate disjunctive cuts for $\P \setminus \Q$.

This paper is organized as follows. In Section~\ref{subsec:relatedlit}, we review related literature. In Section~\ref{subsec:rc-examples}, we provide motivating examples of reverse convex sets. In Section~\ref{sec:ic}, we review intersection cuts. In Section~\ref{sec:cuts-1}, we present a two-term disjunction that contains $\P \setminus \Q$ and is generated by basic solutions of $\P$ that lie outside of $\Q$. In Section~\ref{sec:cuts-2}, we extend this analysis by presenting a multi-term disjunction for $\P \setminus \Q$ by considering $\recc(\Q)$. We propose extended formulations that can be used to define polyhedral relaxations of the disjunctive terms.

\subsection{Related literature}\label{subsec:relatedlit}

The problem of optimizing a linear function over a reverse convex set is known as linear reverse convex programming (LRCP). \citet{tuy1987} shows that any convex program with multiple reverse convex constraints can be reduced to one with a single reverse convex constraint with the introduction of an additional variable and an additional convex constraint. By reduction from a concave minimization problem, optimizing a linear function over a reverse convex set is NP-hard. \citet{matsui1996} shows that this holds even in special cases restricting the structure of the linear constraints or the convex set $\Q$. Reverse convex optimization problems were first studied from a global optimality perspective in the 1970s (e.g., \citet{bansal1975b,bansal1975}, \citet{hillestad1975}, and \citet{ueing1972}). \citet{hillestad1980b,hillestad1980} presented one of the first cutting plane algorithms for LRCP, though it does not always converge to an optimal solution (e.g., \citet{gurlitz1985}). Numerous algorithms for solving LRCP have been developed. \citet{gurlitz1991} present a partial enumeration procedure, \citet{thuong1984} propose sequentially solving concave minimization problems, and \citet{fulop1990} proposes a cutting plane algorithm that cuts off edges of the polyhedron that do not contain an optimal solution. \citet{bensaad1990} present a cutting plane algorithm to solve LRCP based on level sets, but later showed that it does not converge to a globally optimal solution \cite{bensaad1994}. Branch-and-bound methods from concave minimization literature have also been adapted to solve reverse convex optimization problems (e.g., \citet{horst1988}, \citet{horst1990}, \citet{muu1985}, and \citet{ueing1972}).

\citet{hillestad1980} define the concept of a basic solution for LRCP. They show that the convex hull of the feasible region of LRCP is a polytope if the linear constraints form a polytope and the functions defining the reverse convex constraints are differentiable. \citet{sen1987} extend this result, showing that the closure of the convex hull of any polyhedron intersected with a finite number of reverse convex constraints is a polyhedron.

In cutting plane algorithms for MINLP problems, intersection cuts may be constructed from the basis corresponding to an optimal solution to an LP relaxation of the problem. Infeasible basic solutions of $\P$ within $\Q$ are also candidates for generating intersection cuts for $\P \setminus \Q$ and may yield intersection cuts that are not dominated by those generated from feasible basic solutions. Gomory mixed-integer (GMI) cuts for mixed-integer linear programming behave similarly. In particular, \citet {nemhauser1990} showed that the intersection of GMI cuts from all basic solutions is equivalent to the split closure. This is not true when considering only GMI cuts from basic feasible solutions (e.g., \citet{cornuejols2001}).

Intersection cuts can be generated from any convex set that does not contain feasible points in its interior. In integer programming, these are maximal lattice-free convex sets. More generally, these types of sets are convex $S$-free sets. When considering a fixed basis, intersection cuts generated using a convex set $\Q$ can only be stronger than those produced using a subset of $\Q$. \citet{balas2013} generalize intersection cuts such that inequalities for $\P \setminus \Q$ can be obtained using a more general polyhedron rather than a translated simplicial cone. \citet{glover1974} proposes improved intersection cuts for the special case where $\Q$ is a polyhedron. Intersection cuts omit from the inequality variables corresponding to extreme rays of $\PB$ that lie within the recession cone of $\Q$. \citet{glover1974} uses the polyhedron's recession information to include these terms with negative coefficients, thereby strengthening the cut. A similar strengthening is proposed for polynomial optimization problems by \citet{bienstock2020}.

The idea of improving the intersection cut by considering $\recc(\Q)$ has been studied in the context of minimal valid functions. \citet{dey2010} consider minimal valid functions for a polyhedral $\cl(\Q)$ and note that the minimal valid function for $\PB \setminus \Q$ is the uniquely defined intersection cut if $\xbasis \in \Q$ and $\interior(\recc(\Q)) = \emptyset$. Results from this work were extended by \citet{basu2010} and \citet{basu2011}. \citet{fukasawa2011} use the nonnegativity of integer variables to derive minimal valid inequalities for a mixed-integer set. These inequalities consider recession directions of a relevant convex lattice-free set and thus may contain negative variable coefficients.

Ultimately, standard approaches to generating intersection cuts for $\P \setminus \Q$ require a basic solution of $\P$ that lies within the convex set $\Q$. In this paper, we present a framework for constructing valid inequalities for $\P \setminus \Q$ using a basic solution that lies outside $\cl(\Q)$.

\subsection{Motivating examples}\label{subsec:rc-examples}

In many MINLP problems, a reverse convex set $\P \setminus \Q$ can be derived via a problem reformulation. In this case, the polyhedron $\P$ is a relaxation of the MINLP feasible region, and the set $\Q$ is an open convex set which is known to contain no solutions feasible to the MINLP. We can use the set $\P \setminus \Q$ to derive valid inequalities for the original problem. We motivate our study of reverse convex sets by showing how this structure appears in a variety of MINLP contexts. For all of the examples that follow, the closure of the set $\Q$ we derive is non-polyhedral, or possibly defined by a large number of linear inequalities.

One example of this reverse convex structure appears in difference of convex (DC) functions (e.g., \citet{tuy1986}). A function $f \colon \R^n \rightarrow \R$ is a DC function if there exist convex functions $g, h \colon \R^n \rightarrow \R$ such that $f(x) = g(x) - h(x)$ for all $x \in \R^n$. A DC set can be written as
\begin{align}
	\{x \in \R^n \colon g(x) - h(x) \leq 0 \}. \label{eq:dcset}
\end{align}
Equivalently, we can write \eqref{eq:dcset} as $\proj_x (\dcset)$, where $\dcset \coloneqq \{(x,t) \in \R^n \times \R \colon g(x) - t \leq 0,\ h(x) - t \geq 0 \}$. The convex set $\Q = \{(x,t) \in \R^n \times \R \colon h(x) - t < 0 \}$ contains no points feasible to $\dcset$. \citet{hartman1959} shows the class of DC functions is broad, subsuming all twice continuously differentiable functions.

Reverse convex sets also appear in the context of polynomial optimization. \citet{bienstock2020} consider the set of symmetric matrices representable as the outer-product of a vector with itself: $\{x x^\tp \colon x \in \R^n \}$. Polynomial optimization problems can be reformulated to include the constraint that a square matrix of variables is outer-product representable. \citet{bienstock2020} construct non-polyhedral {\em outer-product-free} sets $\Q$ that do not contain any matrices representable as an outer-product of some vector, and as such are not feasible to the problem. Accordingly, they present families of cuts for $\P \setminus \Q$, where $\P$ is formed by the linear constraints of the problem reformulation. They characterize sets that are {\em maximal} outer-product-free, that is, not contained in any other outer-product-free sets. For the specific case of quadratically constrained programs (QCPs), \citet{saxena2010} use disjunctive programming techniques to derive valid inequalities for a reverse convex set in an extended variable space. In a companion paper, \citet{saxena2011} suggest the following {\em eigen-reformulation} of the quadratic constraint $x^\tp A x + a^\tp x + b \leq 0$:
\begin{align*}
	\mash{\sum\limits_{j \colon \lambda_j > 0}}\ &\lambda_j (v_j^\tp x)^2 + \mash{\sum\limits_{j \colon \lambda_j < 0}} \lambda_j s_j + a^\tp x + b \leq 0 \\
	s_j =\ &(v_j^\tp x)^2 \quad \forall j \colon \lambda_j < 0,
\end{align*}
where $\lambda_1, \ldots, \lambda_n$ denote the eigenvalues of $A$ and $v_1, \ldots, v_n$ the corresponding eigenvectors. The convex set $\{ (x,s) \in \R^n \times \R^n \colon s_j > (v_j^\tp x)^2\ \forall j \textrm{ s.t. } \lambda_j < 0\}$ does not contain any points feasible to QCP.

Reverse convex sets can also be used to define relaxations of bilevel optimization problems. Bilevel programs include constraints of the form $d^\tp y \leq \mathit{\Phi}(x)$, where $\mathit{\Phi}(x)$ is the {\em value function} of the lower-level problem for a fixed top-level decision $x$:
\begin{align*}
	\mathit{\Phi}(x) &\coloneqq \min_{y} \{ d^\tp y \colon Ax + By \leq b\}.
\end{align*}
The function $\mathit{\Phi}(\cdot)$ is convex. The set $\{(x,y) \colon d^\tp y > \mathit{\Phi}(x)\}$ is defined by a reverse convex inequality and does not contain any points feasible to the bilevel program. The closure of this set is polyhedral, but may be defined by a large number of linear inequalities. \citet{fischetti2016} propose intersection cuts for a specific class of bilevel integer programming problems.

\section{Intersection cut review}\label{sec:ic}

We briefly review intersection cuts, following the presentation of \citet{conforti2014}. Let $A \in \R^{m \times n}$ be a matrix with full row rank and let $b \in \R^m$. Let $\P = \{x \in \R^n_{+} \colon Ax = b\}$ be a polyhedron. Let $\Q \subseteq \R^n$ be an open convex set. We are interested in valid inequalities for the reverse convex set $\P \setminus \Q$.

For a basis $\B$ of $\P$, let $N = \{1, \ldots, n\} \setminus \B$ be the nonbasic variables. For some $\abar \in \R^{|\B| \times |N|}$ and $\bbar \in \R^{|\B|}_{+}$, we can rewrite $\P$ as
\begin{align*}
	\P &= \Big\lbrace x \in \R^n \colon x_i = \bbar_i - \sum_{j \in N} \abar_{ij} x_j, i \in \B,\ x_j \geq 0, j = 1, \ldots, n \Big\rbrace.
\end{align*}
The basic solution corresponding to basis $\B$ is $\xbasis$, where $\xbasis_i = \bbar_i$ if $i \in B$, and $0$ if $i \in N$. By removing the nonnegativity constraints on variables $x_i$, $i \in \B$, we obtain $\PB$, the cone admitted by the basis $\B$. The basic solution $\xbasis$ forms the apex of $\PB \supseteq \P$. There is an extreme ray $\rbar^j$ of $\PB$ for each $j \in N$:
\begin{align*}
	\rbar^j_k &=
	\begin{cases}
		-\abar_{kj} & \textrm{if } k \in B \\
		1 & \textrm{if } k = j \\
		0 & \textrm{if } k \in N \setminus \{j\}.
	\end{cases}
\end{align*}
The conic hull of the extreme rays $\{\rbar^j \colon j \in N\}$ forms the recession cone of $\PB$. Together, the basic solution $\xbasis$ and these extreme rays provide a complete internal representation of $\PB$, namely, $\PB = \{ \xbasis + \sum_{j \in N} x_j \rbar^j \colon x \in \R^{|N|}_{+} \}$.

Intersection cuts are valid inequalities for $\PB \setminus \Q$ constructed from basic solutions of $\P$ that lie within $\Q$. These cuts are transitively valid for $\P \setminus \Q \subseteq \PB \setminus \Q$. Assume $\xbasis \in \Q$. For each $j \in N$, let $\exit_j$ be defined as
\begin{align}\label{eq:ic-exit-def}
	\exit_j &\coloneqq \sup\{ \exit \geq 0 \colon \xbasis + \exit \rbar^j \in \Q \}.
\end{align}
The set $\{\xbasis + \exit_j \rbar^j \colon j \in N \}$ is the set of points where the extreme rays of $\PB$ emanating from $\xbasis$ leave the set $\Q$. Because $\Q$ is open, $\exit_j > 0$ for all $j \in N$. If $\exit_j = +\infty$, $\rbar^j$ lies in the recession cone of $\Q$.

The following inequality is valid for $\P \setminus \Q$ (\citet{balas1971}):
\begin{align}
	\sum\limits_{j \in N} \frac{x_j}{\exit_j} \geq 1. \label{eq:int-cut-inequality}
\end{align}
We refer to \eqref{eq:int-cut-inequality} as the {\em standard intersection cut}. Here, and throughout the paper, we use the convention that $x / \pm\infty \coloneqq 0$.

\textbf{Notation.} Let $\extreals \coloneqq \R \cup \{-\infty, +\infty\}$ be the extended real numbers. For a nonzero vector $r \in \R^n$ and $\enter,\exit \in \extreals$, we define the line segment $\intervalr{\enter}{\exit}{} \coloneqq \{\lambda r \colon \lambda \in (\enter, \exit) \}$. Closed brackets (e.g., $\clintervalr{\enter}{\exit}{}$) denote the inclusion of one or both endpoints of the line segment. The set $\intervalr{\enter}{\exit}{}$ is unbounded if and only if $\enter = -\infty$ or $\exit = +\infty$. We remark that $\lcintervalnor{0}{+\infty}r$ is equivalent to $\cone(r)$. However, we use the notation $\lcintervalnor{0}{+\infty}r$ for consistency.

\section{Intersection disjunctions and external intersection cuts}\label{sec:cuts-1}

\phantom{}In this paper, we consider a fixed basis $B$ and corresponding basic solution $\xbasis$. Let $\PB$ be defined as in Section~\ref{sec:ic}. For the remainder of this paper, we assume the basic solution $\xbasis$ lies outside of $\cl(\Q)$. Recall $\PB$ is a translated simplicial cone with apex $\xbasis$ and linearly independent extreme rays $\{\rbar^j \colon j \in N\}$. For all $j \in N$, let $\exit_j$ be defined as in \eqref{eq:ic-exit-def}, and let
\begin{alignat*}{2}
	\enter_j &\coloneqq \inf\{ \enter \geq 0 \colon \xbasis + \enter \rbar^j \in \Q \}.
\end{alignat*}
For $j \in N$, we use the convention $\enter_j = +\infty$ and $\exit_j = -\infty$ if the set $\{\xbasis\} + \lcinterval{0}{+\infty}{j}$ does not intersect $\Q$. If $\cl(\Q)$ is polyhedral and $\{\xbasis\} + \lcinterval{0}{+\infty}{j}$ intersects $\Q$, $\enter_j$ and $\exit_j$ can be obtained by solving a linear program. If $\cl(\Q)$ is non-polyhedral, a convex program may be required to obtain these parameters. An exception is the case where $\Q$ is bounded and a point in $\Q \cap (\{\xbasis\} + \lcinterval{0}{+\infty}{j})$ is known a priori, in which case a binary search can be performed to find the values of $\enter_j$ and $\exit_j$.

We partition $N$ into the following three sets:
\begin{align*}
	\N0 &\coloneqq \{j \in N \colon \enter_j = +\infty, \exit_j = -\infty\} \\
	\N1 &\coloneqq \{j \in N \colon \enter_j \in (0, +\infty), \exit_j = +\infty\} \\
	\N2 &\coloneqq \{j \in N \colon \enter_j \in (0, +\infty), \exit_j \in (\enter_j, +\infty)\}.
\end{align*}
For $j \in \N0$, the halfline $\{\xbasis\} + \lcinterval{0}{+\infty}{j}$ does not intersect $\Q$.
Observe $\rbar^j \in \recc(\Q)$ for $j \in \N1$.


Throughout Section~\ref{sec:cuts-1}, we make the following assumption.
\begin{assumption}\label{ass:1}
	It holds that $\rbar^j \in \recc(\Q)$ for all $j \in \N0$.
\end{assumption}
 We say a disjunction is {\em valid} for a set if the disjunction contains the set. Theorem~\ref{thm:n1-n2} proposes a valid disjunction for $\PB \setminus \Q \supseteq \P \setminus \Q$.
\begin{theorem}\label{thm:n1-n2}
	Under Assumption~\ref{ass:1}, for every $x \in \PB \setminus \Q$, either
	\begin{align}
		\sum_{j \in N} \frac{x_j}{\enter_j} &\leq 1, \textrm{ or } \sum_{j \in N} \frac{x_j}{\exit_j} \geq 1. \label{eq:dis}
	\end{align}
\end{theorem}
\begin{proof}
	If $N = \N0$, then $\enter_j = +\infty$ for all $j \in N$, and all $x \in \PB \setminus \Q$ trivially satisfy $\sum_{j \in N} x_j / \enter_j \leq 1$. We prove the result for $N \neq \N0$. Assume $\xhat \in \PB$ satisfies $u \coloneqq \sum_{j \in N} \xhat_j /\enter_j > 1$ and $\ell \coloneqq \sum_{j \in N} \xhat_j / \exit_j < 1$. We show $\xhat \in \Q$.

	Because $u > 1$ and $\ell < 1$, there exists $\gamma \in (0, 1)$ such that $\gamma u + (1 - \gamma)\ell = 1$. For all $j \in N$, let $\theta_j \coloneqq \gamma \xhat_j / \enter_j + (1 - \gamma) \xhat_j / \exit_j \in [0, 1]$. It holds that $\theta_j = 0$ if and only if $j \in \N0$ or $\xhat_j = 0$. Therefore, $\sum_{j \in \N1 \cup \N2 \colon \xhat_j > 0} \theta_j = 1$. We write $\xhat$ as
	\begin{align}
		\xhat &= \ \mash{\sum\limits_{\substack{j \in \N1 \cup \N2 \colon \\ \xhat_j > 0}}} \theta_j \bigg( \xbasis + \frac{\xhat_j}{\theta_j} \rbar^j \bigg) + \sum\limits_{j \in \N0} \xhat_j \rbar^j. \label{eq:thm1-case1}
	\end{align}
	Consider $j \in \N1 \cup \N2$ satisfying $\xhat_j > 0$. If $j \in \N1$, then $\xhat_j / \theta_j = \enter_j / \gamma \in (\enter_j, \exit_j)$. Similarly, if $j \in \N2$, then $\xhat_j / \theta_j = \enter_j \exit_j / (\gamma \exit_j + (1 - \gamma) \enter_j) \in (\enter_j, \exit_j)$. In both cases, $\xbasis + (\xhat_j / \theta_j) \rbar^j \in \Q$.

	By \eqref{eq:thm1-case1}, $\xhat$ is a convex combination of points in $\Q$ plus an element of $\recc(\Q)$. Thus, $\xhat \in \Q$.
\end{proof}

\begin{remark}\label{remark:0}
	The two-term disjunction \eqref{eq:dis} can be used in a disjunctive framework to generate valid inequalities for $\P \setminus \Q$. Specifically, the set $\P \setminus \Q$ is a subset of $\P_1 \cup \P_2$, where $\P_1 \coloneqq \{x \in \P \colon \sum_{j \in N} x_j / \enter_j \leq 1\}$ and $\P_2 \coloneqq \{x \in \P \colon \sum_{j \in N} x_j / \exit_j \geq 1\}$. The sets $\P_1$ and $\P_2$ are polyhedral, because $\P$ is a polyhedron and the inequalities \eqref{eq:dis} are linear. We can obtain valid inequalities for $\conv(\P \setminus \Q)$ by generating valid inequalities for $\conv(\P_1 \cup \P_2)$ using the disjunctive programming approach of \citet{balas1979,balas1974}.
\end{remark}

\begin{remark}\label{remark:1}
	We consider the relationship between the two-term disjunction \eqref{eq:dis} and the standard intersection cut. The two-term disjunction \eqref{eq:dis} assumes that the basic solution $\xbasis$ does not lie within $\cl(\Q)$. If $\xbasis \in \Q$, then $\N0 = \emptyset$ (trivially, every extreme ray of $\PB$ emanating from $\xbasis \in \Q$ intersects $\Q$) and $\enter_j = 0$ for all $j \in N$. Because $\enter_j = 0$ for all $j \in N$, the inequality $\sum_{j \in N} x_j / \enter_j \leq 1$ of \eqref{eq:dis} is ill-defined. Instead, we can show that all points in $\PB \setminus \Q$ lie in either $\{\xbasis\}$ or $\{x \in \R^n \colon \sum_{j \in N} x_j / \exit_j \geq 1\}$. However, because $\{\xbasis\} \subseteq \Q$, we can conclude that the inequality $\sum_{j \in N} x_j / \exit_j \geq 1$ is valid for $\PB \setminus \Q$. This is precisely the standard intersection cut of \citet{balas1971}.
\end{remark}

\begin{example}\label{ex:pq}
	Let $\P = \R_{+}^2$ and $\Q = \{ x \in \R^2 \colon (x_1 - 1)^2 - x_2 < 1/2 \}$. Consider $\PB$ generated by the (only) basic solution of $\P$, $\xbasis = (0,0) \notin \Q$. In this case, $\PB = \P$. The feasible region $\P \setminus \Q$ is the disconnected set shaded in Figure~\ref{fig:t-1}. The inequalities \eqref{eq:dis} form a valid disjunction for $\P \setminus \Q$, shown in Figure~\ref{fig:two-way-dis}.
	
	\begin{figure}
		\begin{subfigure}[t]{0.46\textwidth}
			\vspace*{0pt}
			\centering
			\includegraphics[width=39mm]{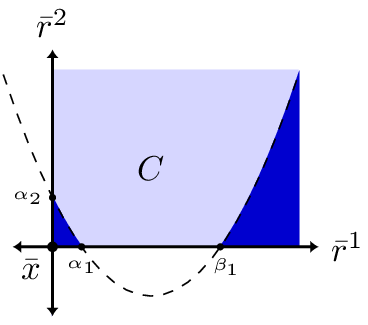}
			\caption{The set $\P \setminus \Q$ for Example~\ref{ex:pq} is the darkened, disconnected region.}
			\label{fig:t-1}
		\end{subfigure}%
		\hfill
		\begin{subfigure}[t]{0.46\textwidth}
			\vspace*{0pt}
			\centering
			\includegraphics[width=39mm]{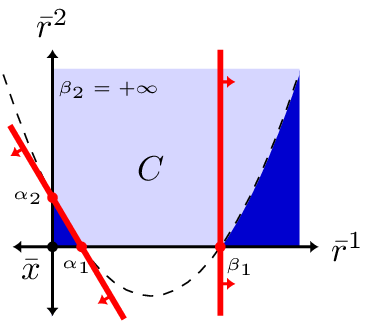}
			\caption{Every point in the darkened set $\P \setminus \Q$ satisfies one of the inequalities \eqref{eq:dis}.}
			\label{fig:two-way-dis}
		\end{subfigure}
		\caption{The two-term disjunction of Theorem~\ref{thm:n1-n2} applied to Example~\ref{ex:pq}.}
	\end{figure}
\end{example}

Proposition~\ref{prop:termwise-conv} states that if $\Q$ is bounded, then the inequality defining each term of \eqref{eq:dis} is sufficient to define the convex hull of the points in $\PB \setminus \Q$ satisfying that inequality. This is not true if the interior of $\recc(\Q)$ is nonempty, as shown by \citet{dey2010} for the standard intersection cut \eqref{eq:int-cut-inequality}.
\begin{proposition}\label{prop:termwise-conv}
	Under Assumption~\ref{ass:1}, if $\Q$ is bounded, then 
	\begin{align*}
		\conv(\{x \in \PB \setminus \Q \colon \textstyle\sum_{j \in N} x_j / \enter_j \leq 1\}) &= \{x \in \PB \colon \textstyle\sum_{j \in N} x_j / \enter_j \leq 1\}, \textrm{ and} \\
		\conv(\{x \in \PB \setminus \Q \colon \textstyle\sum_{j \in N} x_j / \exit_j \geq 1\}) &= \{x \in \PB \colon \textstyle\sum_{j \in N} x_j / \exit_j \geq 1\}.
	\end{align*}
\end{proposition}
\begin{proof}
	We show only that $\conv(\{x \in \PB \setminus \Q \colon \textstyle\sum_{j \in N} x_j / \enter_j \leq 1\}) = \{x \in \PB \colon \textstyle\sum_{j \in N} x_j / \enter_j \leq 1\}$, as the second statement can be shown using similar techniques. Under Assumption~\ref{ass:1}, $\Q$ bounded implies $N = \N2$.
	
	Because $\PB \setminus \Q \subseteq \PB$ and the set $\{x \in \PB \colon \textstyle\sum_{j \in N} x_j / \enter_j \leq 1\}$ is convex, $\conv(\{x \in \PB \setminus \Q \colon \textstyle\sum_{j \in N} x_j / \enter_j \leq 1\}) \subseteq \{x \in \PB \colon \textstyle\sum_{j \in N} x_j / \enter_j \leq 1\}$. Next, let $\xhat \in \PB$ satisfy $\sum_{j \in N} \xhat_j / \enter_j \leq 1$. Then
	\begin{align}
		\xhat &= \xbasis + \sum\limits_{j \in N} \xhat_j \rbar^j = \sum\limits_{j \in N} \frac{\xhat_j}{\enter_j} (\xbasis + \enter_j \rbar^j) + \bigg( 1 - \sum\limits_{j \in N} \frac{\xhat_j}{\enter_j} \bigg) \xbasis \nonumber \\
		&\in \conv(\{\xbasis\} \cup \{\xbasis + \enter_j \rbar^j \colon j \in N\}) \subseteq \conv(\P \setminus \Q). \label{eq:conv-points-ineq}
	\end{align}
	Because $\xbasis_j = 0$ for all $j \in N$, we have $\sum_{j \in N} \xbasis_j / \enter_j = 0$. For $i \in N$, let $y^i_j \coloneqq \xbasis + \enter_i \rbar^i$. For any $i,j \in N$, $y^i_j$ equals $\enter_j$ if $i = j$, and $0$ otherwise. Hence, for all $i \in N$, $\sum_{j \in N} y^i_j / \enter_j = 1$. Continuing from \eqref{eq:conv-points-ineq}, we have $\xhat \in \conv(\{x \in \PB \setminus \Q \colon \textstyle\sum_{j \in N} x_j / \enter_j \leq 1\})$.
\end{proof}

The disjunction presented in Theorem~\ref{thm:n1-n2} can be particularly useful if $\P$ is empty when intersected with one of the inequalities \eqref{eq:dis}. In this case, the inequality defining the other disjunctive term is valid for $\PB \setminus \Q$.
\begin{definition}
	If $\{x \in \P \colon \sum_{j \in N} x_j / \exit_j \geq 1\} = \emptyset$, we refer to the inequality $\sum_{j \in N} x_j / \enter_j \leq 1$ as an {\em external intersection cut}. We say the same for the inequality $\sum_{j \in N} x_j / \exit_j \geq 1$ if $\{x \in \P \colon \sum_{j \in N} x_j / \enter_j \leq 1\} = \emptyset$.
\end{definition}
External intersection cuts are valid for $\P \setminus \Q$. We provide an example where intersection cuts are insufficient to define $\conv(\P \setminus \Q)$, but the facet-defining inequality for $\conv(\P \setminus \Q)$ can be obtained from an external intersection cut. In order to derive the inequalities \eqref{eq:dis}, we must first translate our polyhedral set to standard form, using additional slack variables as necessary. We then select a basis and calculate $\enter_j$ and $\exit_j$ for all $j \in N$. In this example and all that follow, we intentionally omit the intermediate steps required to obtain these inequalities, presenting them in the original variable space.
\begin{example}\label{ex:intro}		
	Let
	\begin{align*}
		\P &= \{ x \in \R_{+}^2 \colon -x_1 + 3x_2 \leq 3/2 \} \\
		\Q &= \{ x \in \R^2 \colon ||x||_2 < 1 \}.
	\end{align*}
	As can be seen in Figure~\ref{fig:unobtainable}, no standard intersection cut is able to generate the inequality that is facet-defining for $\conv(\P \setminus \Q)$. However, the basic solution $\xbasis = (-3/2, 0) \notin \cl(\Q)$ corresponding to the constraints $x_2 \geq 0$ and $-x_1 + 3x_2 \leq 3/2$ generates this inequality as an external intersection cut. For this basic solution, the set $\{x \in \P \colon \sum_{j \in N} x_j / \enter_j \leq 1\}$ is empty, implying that the inequality $\sum_{j \in N} x_j / \exit_j \geq 1$ is valid for $\P \setminus \Q$. Figure~\ref{fig:obtainable} shows the inequalities \eqref{eq:dis} for this example.
	
	We note that in this example, there does exist an open convex set $\Q' \supseteq \Q$ such that $\xbasis \in \Q'$ and the intersection cut defined by $\xbasis$ with respect to $\Q$ generates the facet-defining inequality for $\conv(\P \setminus \Q)$. For instance, one such set is $\Q' = \Q \cup \{ x \in \R^2 \colon -1 < x_2 < 1,\ x_1 < 0\}$. Methods for enlarging the set $\Q$ to generate intersection cuts stronger than those produced by $\Q$ are outside the scope of this work, though this topic has been studied by \citet{balas1972}.
	
	\begin{figure}
		\begin{subfigure}[t]{0.44\textwidth}
			\vspace*{0pt}
			\centering
			\includegraphics[height=40mm]{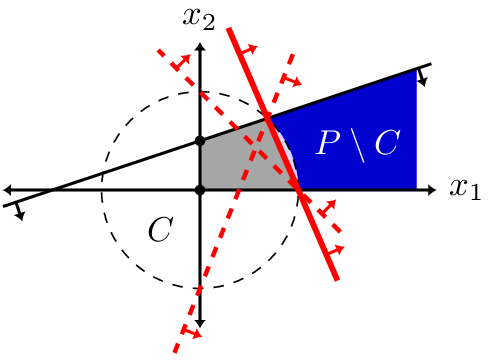}
			\caption{The facet-defining inequality (solid line) for the reverse convex set in Example~\ref{ex:intro} is {\em not} obtainable as a standard intersection cut (dashed lines) from one of the two basic solutions of $\P$ that lie within $\Q$.}
			\label{fig:unobtainable}
		\end{subfigure}
		\begin{subfigure}[t]{0.05\textwidth}
			\quad
		\end{subfigure}%
		\begin{subfigure}[t]{0.44\textwidth}
			\vspace*{0pt}
			\centering
			\includegraphics[height=40mm]{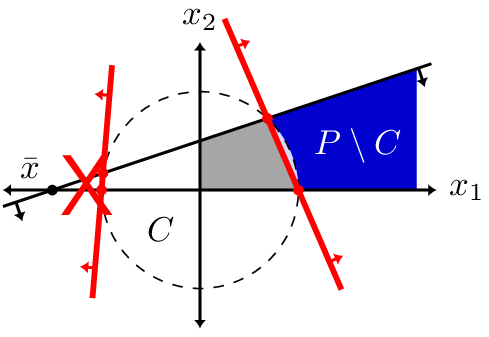}
			\caption{An external intersection cut from the basic solution $\xbasis = (-3/2, 0)$ defines the facet-defining inequality for $\conv(\P \setminus \Q)$ for Example~\ref{ex:intro}. The set $\{x \in \P \colon \sum_{j \in N} x_j / \enter_j \leq 1\}$ is empty, so the inequality $\sum_{j \in N} x_j / \exit_j \geq 1$ is valid for $\P \setminus \Q$.}
			\label{fig:obtainable}
		\end{subfigure}
		\caption{An external intersection cut for Example~\ref{ex:intro}.}
	\end{figure}
\end{example}

\begin{example}\label{ex:consequence}
	Figure~\ref{fig:consequence} depicts another example of an external intersection cut. The extreme rays of $\PB$ enter into the convex set $\Q$ and remain within $\Q$ on an unbounded interval, so $\{ x \in \P \colon \sum_{j \in N} x_j / \exit_j \geq 1 \} = \emptyset$. The external intersection cut $\sum_{j \in N} x_j / \enter_j \leq 1$ is valid for $\P \setminus \Q$.
	\begin{figure}[ht]
		\centering
		\includegraphics[width=46mm]{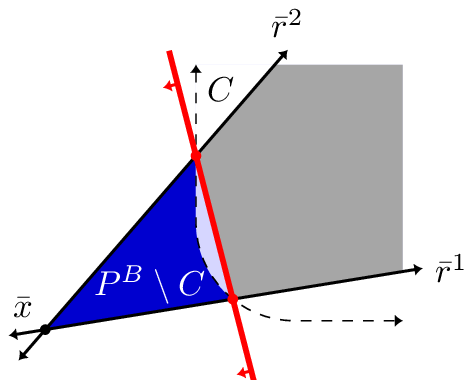}
		\caption{The external intersection cut $\sum_{j \in N} x_j / \enter_j \leq 1$ for Example~\ref{ex:consequence} arises when $\{x \in \P \colon \sum_{j \in N} x_j / \exit_j \geq 1\} = \emptyset$. Because the extreme rays of the translated simplicial cone $\PB$ recede into $\Q$, $\exit_j = +\infty$ for all $j \in N$, and hence $\sum_{j \in N} x_j / \exit_j = 0$.}
		\label{fig:consequence}
	\end{figure}
\end{example}
\begin{definition}
	If $\{ x \in \P \colon \sum_{j \in N} x_j / \exit_j \geq 1 \} \neq \emptyset$ and $\{ x \in \P \colon \sum_{j \in N} x_j / \enter_j \leq 1 \} \neq \emptyset$, we say the disjunction \eqref{eq:dis} is an {\em intersection disjunction} for $\P \setminus \Q$.
\end{definition}
If \eqref{eq:dis} is an intersection disjunction for $\P \setminus \Q$, we can use a disjunctive CGLP to generate valid inequalities for $\conv (\P \setminus \Q)$ using the techniques of \citet{balas1979,balas1974}.

We provide an example of why Assumption~\ref{ass:1} is necessary for the validity of the two-term disjunction of Theorem~\ref{thm:n1-n2}.
\begin{example}\label{ex:ass1}
	Consider the reverse convex set shown in Figure~\ref{fig:n0-1}. The extreme ray $\rbar^2$ does not intersect the bounded convex set $\Q$, so \eqref{eq:dis} is not a disjunction for $\P \setminus \Q$. Figure~\ref{fig:n0-2} shows the same example but with the halfline $\lcinterval{0}{+\infty}{2}$ added to the set $\Q$. Because Assumption~\ref{ass:1} holds, Theorem~\ref{thm:n1-n2}'s disjunction is valid for $\P \setminus \Q$.
	
	\begin{figure}
		\vspace*{0pt}
		\begin{subfigure}[t]{0.44\textwidth}
			\vspace*{0pt}
			\centering
			\includegraphics[width=46mm]{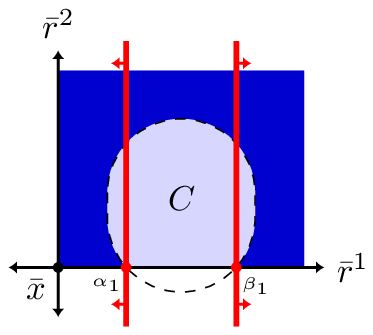}
			\caption{Theorem~\ref{thm:n1-n2}'s disjunction is \emph{not} valid if there exists $j \in \N0$ such that $\rbar^j \notin \recc(\Q)$. In this case, $\rbar^2 \in \N0$, but $\rbar^2 \notin \recc(\Q)$. There are points in the darkened region $\P \setminus \Q$ that satisfy neither of the two inequalities shown.}
			\label{fig:n0-1}
		\end{subfigure}%
		\begin{subfigure}[t]{0.05\textwidth}
			\quad
		\end{subfigure}%
		\begin{subfigure}[t]{0.44\textwidth}
			\vspace*{0pt}
			\centering
			\includegraphics[width=46mm]{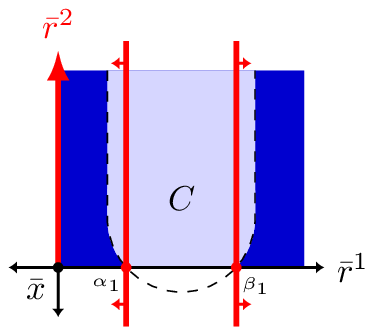}
			\caption{The validity of Theorem~\ref{thm:n1-n2}'s disjunction relies on Assumption~\ref{ass:1}. If we modify $\Q$ so $\rbar^2 \in \recc(\Q)$, the disjunction of Theorem~\ref{thm:n1-n2} is valid.}
			\label{fig:n0-2}
		\end{subfigure}
		\caption{Assumption~\ref{ass:1} is necessary for our analysis.}
	\end{figure}
\end{example}

Our final example of this section motivates considering how the recession cone can be used to derive more general valid disjunctions for $\PB \setminus \Q$.
\begin{example}\label{ex:useless}
	Let $\P = \R^2_{+}$, and
	\begin{align*}
		\Q &= \Big\lbrace (x_1,x_2) \colon \Big(x_1 - \frac{3}{4} \Big)^2 + \Big(x_2 - \frac{1}{4} \Big)^2 < \frac{1}{4} \Big\rbrace + \cone\Big( \begin{bmatrix} 1 \\ 1 \end{bmatrix}, \begin{bmatrix} 2 \\ 1 \end{bmatrix} \Big).
	\end{align*}
	Figure~\ref{fig:n0-insufficient-mod} provides a graphical representation of $\PB \setminus \Q$, where $\PB$ is generated from the basic solution $\xbasis = (0,0)$. Assumption~\ref{ass:1} does not hold; namely, $2 \in \N0$, but $\rbar^2 \notin \recc(\Q)$. However, there exists a valid two-term disjunction for $\PB \setminus \Q$ that cannot be obtained with the theory of this section.
	\begin{figure}[ht]
		\centering
		\includegraphics[width=46mm]{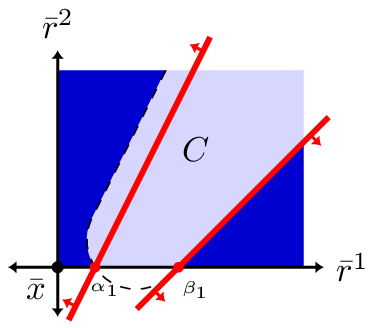}
		\caption{Although a valid two-term disjunction for Example~\ref{ex:useless} exists, the inequalities of Theorem~\ref{thm:n1-n2} are insufficient to obtain it, because Assumption~\ref{ass:1} is not satisfied.}
		\label{fig:n0-insufficient-mod}
	\end{figure}
\end{example}

\section{Valid inequalities and intersection disjunctions using \texorpdfstring{$\recc(\Q)$}{recc(C)}}\label{sec:cuts-2}

In this section, we generalize the results of Section~\ref{sec:cuts-1} by considering the full recession cone of $\Q$. In Section~\ref{subsec:inner}, we construct an inner approximation of $\Q$ and analyze its relationship to $\PB \setminus \Q$. We derive inequalities to define a polyhedral relaxation of $\PB \setminus \Q$ in Section~\ref{subsec:tq-cuts}. In Section~\ref{subsec:multiterm}, we generalize the two-term disjunction of Theorem~\ref{thm:n1-n2} to a multi-term disjunction that uses the recession cone of $\Q$. We propose polyhedral relaxations of these disjunctive terms in Section~\ref{subsec:skq-cuts-combined}.

\subsection{An inner approximation of \texorpdfstring{$\Q$}{C}}\label{subsec:inner}

Let $\Nonetwo \coloneqq \N1 \cup \N2$. We define $\T$ and $\T^\Q$ as follows:
\begin{align}
	\T &\coloneqq \{\xbasis\} + \conv \big(\textstyle\bigcup_{j \in \Nonetwo} \interval{\enter_j}{\exit_j}{j} \big),\enspace \T^\Q \coloneqq \T + \recc(\Q). \label{eq:def-tq}
\end{align}
Both $\T$ and $\T^\Q$ are subsets of $\Q$. Additionally, we define $\Rset$ and $\Rset^\Q$ as follows:
\begin{align*}
	\Rset &\coloneqq \{\xbasis\} + \conv \big( \textstyle\bigcup_{j \in \N1} \interval{\enter_j}{\exit_j}{j} \big),\enspace \Rset^\Q \coloneqq \Rset + \recc(\Q).
\end{align*}
Note that $\Rset^\Q \subseteq \T^\Q$. We derive inequalities valid for $\PB \setminus \T^\Q \supseteq \PB \setminus \Q$. We illustrate the sets $\T^\Q$ and $\Rset^\Q$ graphically in the example that follows.
\begin{example}\label{ex:tq}
	Let $\P = \R^2_{+}$ and $\Q = \{(x_1,x_2) \colon x_2 > \sqrt{(x_1 - 1)^2+1} - 1.1 \}$. Let $\xbasis = (0,0)$ be the basic solution of $\P$, corresponding to basis $B$. The sets $\PB$ and $\Q$ are shown in Figure~\ref{fig:pb-tq-1}. Figures \ref{fig:pb-tq-2} and~\ref{fig:pb-tq-3} show the sets $\T^\Q$ and $\PB \setminus \T^\Q$, respectively. Figure~\ref{fig:pb-rq-1} shows the set $\Rset^\Q$. The set $\PB \setminus \Rset^\Q$ is depicted in Figure~\ref{fig:pb-rq-2}.
	
	\begin{figure}
		\begin{subfigure}[t]{0.3\textwidth}
			\vspace*{0pt}
			\centering
			\includegraphics[width=46mm]{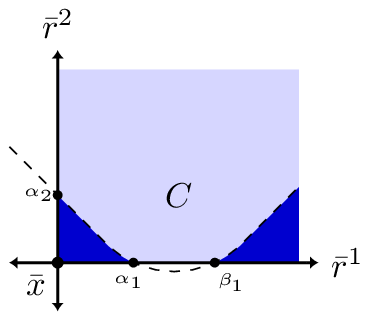}
			\caption{The set $\PB \setminus \Q$ for Example~\ref{ex:tq} is the darkened, disconnected region.}
			\label{fig:pb-tq-1}
		\end{subfigure}%
		\begin{subfigure}[t]{0.05\textwidth}
			\quad
		\end{subfigure}%
		\begin{subfigure}[t]{0.29\textwidth}
			\vspace*{0pt}
			\centering
			\includegraphics[width=46mm]{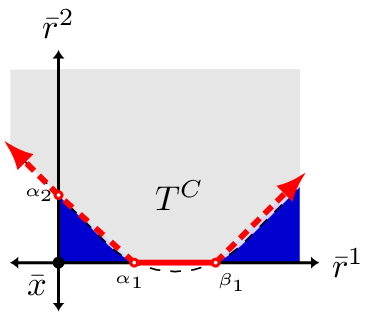}
			\caption{The set $\T^\Q$ is an inner approximation of $\Q$.}
			\label{fig:pb-tq-2}
		\end{subfigure}%
		\begin{subfigure}[t]{0.05\textwidth}
			\quad
		\end{subfigure}%
		\begin{subfigure}[t]{0.3\textwidth}
			\vspace*{0pt}
			\centering
			\includegraphics[width=46mm]{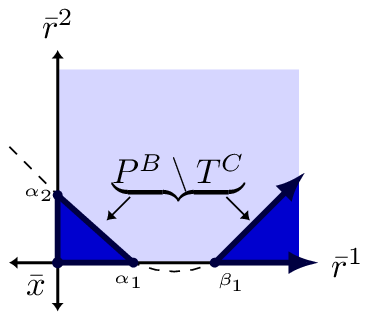}
			\caption{The darkened and disconnected set $\PB \setminus \T^\Q$ is a relaxation of $\PB \setminus \Q$.}
			\label{fig:pb-tq-3}
		\end{subfigure}
		\caption{The construction of $\PB \setminus \T^\Q$ for Example~\ref{ex:tq}. \label{fig:pb-tq}}
	\end{figure}
\end{example}

\begin{figure}
	\begin{subfigure}[t]{0.44\textwidth}
		\vspace*{0pt}
		\centering
		\includegraphics[width=46mm]{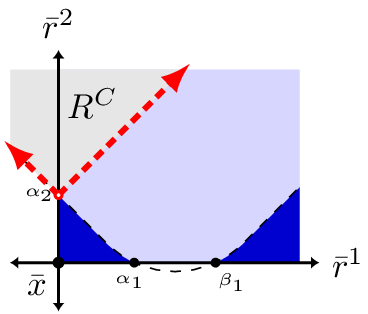}
		\caption{The set $\Rset^\Q$ is an inner approximation of $\Q$ that does not consider points along extreme rays of $\PB$ corresponding to indices in $\N2$.}
		\label{fig:pb-rq-1}
	\end{subfigure}%
	\begin{subfigure}[t]{0.05\textwidth}
		\quad
	\end{subfigure}%
	\begin{subfigure}[t]{0.44\textwidth}
		\vspace*{0pt}
		\centering
		\includegraphics[width=46mm]{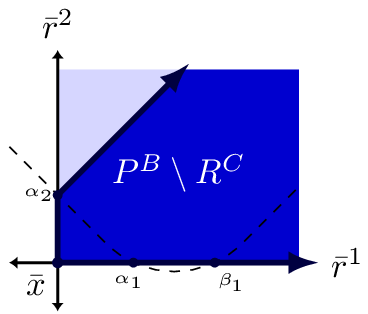}
		\caption{The darkened set $\PB \setminus \Rset^\Q$ is a relaxation of $\PB \setminus \Q$. Its relationship to $\PB \setminus \Q$ and $\PB \setminus \T^\Q$ is established in Theorem~\ref{thm:conv-cq}.}
		\label{fig:pb-rq-2}
	\end{subfigure}
	\caption{The construction of $\Rset^\Q$ for Example~\ref{ex:tq}.}
\end{figure}

We motivate the study of $\PB \setminus \T^\Q$ and $\PB \setminus \Rset^\Q$ by showing that each set retains the strength of $\PB \setminus \Q$ under the convex hull operator.
\begin{theorem}\label{thm:conv-cq}
	It holds that
	\begin{align*}
		\clconv(\PB \setminus \Q) = \clconv(\PB \setminus \T^\Q) = \clconv(\PB \setminus \Rset^\Q).
	\end{align*}
\end{theorem}
\begin{proof}
	Observe $\{\xbasis\} + \bigcup_{j \in \N1} \interval{\enter_j}{\exit_j}{j} \subseteq \{\xbasis\} + \bigcup_{j \in \Nonetwo} \interval{\enter_j}{\exit_j}{j} \subseteq \Q$. By definition, $\Rset^\Q \subseteq \T^\Q \subseteq \Q$, which implies
	\begin{align*}
		\PB \setminus \Q \subseteq \PB \setminus \T^\Q \subseteq \PB \setminus \Rset^\Q.
	\end{align*}
	To complete the proof, we show $\PB \setminus \Rset^\Q \subseteq \clconv(\PB \setminus \Q)$. Let $y \in \PB \setminus \Rset^\Q$. Assume $y \in \Q$, or we have nothing to prove. Because $y \in \PB$, we have $y = \xbasis + \sum_{j \in N} y_j \rbar^j$, where $y_j \geq 0$ for all $j \in N$. Let $\eta \coloneqq \sum_{j \in \N1} y_j / \enter_j$.
	
	To begin, assume $\eta < 1$. Assume also that $\sum_{j \in \N0 \cup \N2} y_j > 0$, otherwise $y$ is a convex combination of the points $\{\xbasis\} \cup \{\xbasis + \enter_j \rbar^j \colon j \in \N1\} \subseteq \PB \setminus \Q$:
	\begin{align*}
		y &= (1 - \eta)\xbasis + \sum_{j \in \N1} \frac{y_j}{\enter_j} (\xbasis + \enter_j \rbar^j).
	\end{align*}
	Let $\lambday \coloneqq \sum_{j \in \N0 \cup \N2} y_j / (1 - \eta)$. We rewrite $y$ as
	\begin{align*}
		y &= \sum_{j \in \N1} \frac{y_j}{\enter_j} (\xbasis + \enter_j \rbar^j) + \mash{\sum_{j \in \N0 \cup \N2}} \ \frac{y_j}{\lambday} (\xbasis + \lambday \rbar^j).
	\end{align*}
	We have $\xbasis + \enter_j \rbar^j \in \PB \setminus \Q$ for all $j \in \N1$. Additionally, $\xbasis + \lambday \rbar^j \in \PB \setminus \Q$ for all $j \in \N0$. For $j \in \N2$, $\xbasis + \lambday \rbar^j \in \conv(\PB \setminus \Q)$, because $\xbasis \in \PB \setminus \Q$ and $\xbasis + \delta \rbar^j \in \PB \setminus \Q$ for a sufficiently large $\delta > \lambday$. The coefficients on the vectors $\{\xbasis + \enter_j \rbar^j \colon j \in \N1\} \cup \{\xbasis + \lambday \rbar^j \colon j \in \N0 \cup \N2\}$ are nonnegative and sum to one. Then $y \in \conv (\PB \setminus \Q)$.
	
	Next, assume $\eta = 1$. Because $y_j \geq 0$ for all $j \in N$, we have $\sum_{j \in \N1} y_j > 0$. Then there exists $k \in \N1$ satisfying $y_k > 0$. Let $y^{\epsilon} \coloneqq y - \epsilon \rbar^k$. For all $\epsilon \in (0, y_k]$, it holds that $y^\epsilon \in \PB$. Furthermore, $\sum_{j \in \N1} y^\epsilon_j / \enter_j = \eta - \epsilon / \enter_k < 1$. It follows from the above analysis that $y^\epsilon \in \conv(\PB \setminus \Q)$. Consequently, $y = \lim_{\epsilon \rightarrow 0} y^\epsilon \in \clconv(\PB \setminus \Q)$.
	
	Finally, assume $\eta > 1$. For $\epsilon \in [0,1)$, let $z^\epsilon$ be the following:
	\begin{align*}
		z^\epsilon &\coloneqq \xbasis + \sum_{j \in \N1} \frac{y_j}{\lambdaz} [(1 - \epsilon)\lambdaz + \epsilon] \rbar^j
	\end{align*}
	Consider a fixed $\epsilon \in [0,1)$. It holds that $z^\epsilon \in \Rset$, because
	\begin{align*}
		z^\epsilon =  \xbasis + \sum_{j \in \N1} \frac{y_j / \enter_j}{\lambdaz} [(1 - \epsilon)\lambdaz + \epsilon] \enter_j \rbar^j \in \{\xbasis\} + \conv \big( \textstyle\bigcup_{j \in \N1} \interval{\enter_j}{\exit_j}{j} \big).
	\end{align*}
	It must be the case that $y - z^\epsilon \notin \recc(\Q)$. If not, we have $y = z^\epsilon + q$ for some $q \in \recc(\Q)$, implying $y \in \Rset^\Q$ and contradicting $y \in \PB \setminus \Rset^\Q$.
	
	It holds that $y - z^\epsilon \in \recc(\PB)$, because the coefficients on the terms $\{\rbar^j \colon j \in N\}$ are nonnegative:
	\begin{align*}
		y - z^\epsilon &= \mash{\sum\limits_{j \in \N0 \cup \N2}} y_j \rbar^j  + \sum\limits_{j \in \N1} \frac{y_j}{\lambdaz}(\lambdaz - 1) \epsilon \rbar^j.
	\end{align*}
	Then we have $y - z^\epsilon \in \recc(\PB) \setminus \recc(\Q)$. For a sufficiently large $\gamma_{\epsilon} > 0$, $y + \gamma_{\epsilon}(y - z^\epsilon) \notin \Q$, because $y - z^\epsilon \notin \recc(\Q)$. Because $y - z^\epsilon \in \recc(\PB)$, it follows that $y + \gamma_{\epsilon}(y - z^\epsilon) \in \PB \setminus \Q$. Let $\zhat \in \conv(\PB \setminus \Q)$ be defined as follows:
	\begin{align*}
		\zhat &\coloneqq \xbasis + \sum_{j \in \N1} \frac{y_j}{\lambdaz} \rbar^j = \sum_{j \in \N1} \frac{y_j / \enter_j}{\lambdaz} (\xbasis + \enter_j \rbar^j) \in \conv(\{ \xbasis + \enter_j\rbar^j \colon j \in \N1 \})
	\end{align*}
	For any $\epsilon \in [0,1)$, let $\vv^{\epsilon}$ be the following convex combination of $\zhat$ and $y + \gamma_{\epsilon}(y - z^\epsilon)$:
	\begin{align}
		\vv^{\epsilon} \coloneqq \frac{\gamma_{\epsilon}}{\gamma_{\epsilon} + 1} \zhat + \frac{1}{\gamma_{\epsilon} + 1} \left( y + \gamma_{\epsilon}(y - z^\epsilon) \right) \in \conv(\PB \setminus \Q). \label{eq:zhat-def}
	\end{align}
	We have $\lim_{\epsilon \rightarrow 1} z^\epsilon = \zhat$, and $\gamma_{\epsilon} / (\gamma_{\epsilon} + 1) \in [0,1)$ for all $\gamma_{\epsilon} > 0$. Then
	\begin{align*}
		\lim\limits_{\epsilon \rightarrow 1} \bigg\lvert \bigg\lvert \frac{\gamma_{\epsilon}}{\gamma_{\epsilon} + 1}(\zhat - z^\epsilon) \bigg\rvert \bigg\rvert \leq \lim\limits_{\epsilon \rightarrow 1} \bigg\lvert \frac{\gamma_{\epsilon}}{\gamma_{\epsilon} + 1} \bigg\rvert \lvert\lvert \zhat - z^\epsilon \rvert\rvert &= 0.
	\end{align*}
	Rearranging the definition of $\vv^\epsilon$ from \eqref{eq:zhat-def}, we have
	$y = \vv^\epsilon + [\gamma_{\epsilon} / (\gamma_{\epsilon} + 1)](z^\epsilon - \zhat)$.
	Thus,
	\begin{align*}
		y &= \lim\limits_{\epsilon \rightarrow 1} \vv^{\epsilon} \in \clconv(\PB \setminus \Q). \tag*{\qedhere}
	\end{align*}
	\vspace*{-\belowdisplayskip}
\end{proof}
Theorem~\ref{thm:conv-cq} supports our selection of $\PB \setminus \T^\Q$ as a relaxation of $\PB \setminus \Q$, as we do not lose anything when considering $\clconv (\PB \setminus \T^\Q)$. For the remainder of Section~\ref{sec:cuts-2}, we make the following assumption.
\begin{assumption}\label{ass:recc}
	The recession cone of $\Q$ is contained in the recession cone of $\PB$.
\end{assumption}

If Assumption~\ref{ass:recc} does not hold, we can consider the convex set $\PB \cap \Q$ instead of $\Q$. Indeed, $\recc(\PB \cap \Q) \subseteq \recc(\PB)$. Our analysis only requires the set $\Q$ to be relatively open in $\PB$, not necessarily open. By Corollary~\ref{cor:replace}, replacing $\Q$ with $\PB \cap \Q$ does not change the strength of our relaxation of $\PB \setminus \Q$ with respect to the convex hull operator.
\begin{corollary}\label{cor:replace}
	It holds that $\clconv(\PB \setminus \Q) = \clconv(\PB \setminus \T^{\PB \cap \Q})$.
\end{corollary}
\begin{proof}
	The statement follows directly from Theorem~\ref{thm:conv-cq}:
	\begin{align*}
		\clconv(\PB \setminus \Q) &= \clconv(\PB \setminus (\PB \cap \Q)) = \clconv(\PB \setminus \T^{\PB \cap \Q}). \tag*{\qedhere}
	\end{align*}
\end{proof}
If $\PB$ in Corollary~\ref{cor:replace} is replaced with $\P$, the statement is no longer true in general. This is relevant, because we present a valid disjunction for $\PB \setminus \T^\Q$ in Section~\ref{subsec:multiterm}. When we add the constraints of $\P$ to the disjunctive formulation of $\PB \setminus \T^\Q$, the cuts obtained from a CGLP could be stronger than those obtained from a CGLP built from a valid disjunction for $\PB \setminus \T^{\PB \cap \Q}$. Thus, substituting $\Q$ with $\PB \cap \Q$ in order to satisfy Assumption~\ref{ass:recc} has the potential to weaken the generated disjunctive cuts.

\subsection{Polyhedral relaxation of \texorpdfstring{$\PB \setminus \T^\Q$}{P\carrot B \back T\carrot C}}\label{subsec:tq-cuts}

We present valid inequalities for $\PB \setminus \genset_{\D}^{\Q}$, where $\D \subseteq \Nonetwo$ is fixed and
\begin{align}
	\genset_{\D} &\coloneqq \{\xbasis\} + \conv \big(\textstyle\bigcup_{j \in \D} \interval{\enter_j}{+\infty}{j} \big),\enspace \genset_{\D}^{\Q} \coloneqq \genset_{\D} + \recc(\Q). \label{eq:genset-def}
\end{align}
In this section, we are interested in deriving valid inequalities for $\PB \setminus \genset_{\D}^{\Q}$ when $\D = \N1$, in which case $\genset_{\D}^{\Q} = \Rset^\Q$. Because $\Rset^\Q \subseteq \T^\Q$, inequalities valid for $\PB \setminus \Rset^\Q$ are also valid for $\PB \setminus \T^\Q$. We consider the more general set $\genset_{\D}^{\Q}$ to be able to apply this analysis in Section~\ref{subsec:skq-cuts-combined}. Observe that $\recc(\genset_{\D}^{\Q}) = \cone(\{ \rbar^j \colon j \in \D\}) + \recc(\Q)$. Furthermore, $\genset_{\D}^{\Q} \subseteq \Q$ if and only if $\D \subseteq \N1$.

For $(i,j) \in \D \times (N \setminus \D)$, let $\epshatone_{ij}$ be the following:
\begin{align*}
	\epshatone_{ij} &\coloneqq \sup \{ \epshatone \geq 0 \colon \enter_i \rbar^i + \epshatone \rbar^j \in \recc(\genset_{\D}^{\Q}) \}.
\end{align*}
It holds that $\epshatone_{ij} = +\infty$ if $\rbar^j \in \recc(\genset_{\D}^{\Q})$. In all other cases, $\epshatone_{ij}$ is finite and its supremum is attained, because $\recc(\genset_{\D}^{\Q})$ is a closed convex cone. The parameter $\epshatone_{ij}$ depends on $\D$, but we suppress this dependence for notational simplicity.

Let $\Msetone$ be defined as follows:
\begin{align*}
	\Msetone &\coloneqq \{i \in \D \colon \epshatone_{ij} > 0\ \forall j \in N \setminus \D\}.
\end{align*}
Let $\F_{ij} \coloneqq \cone(\rbar^i,\rbar^j)$ be the cone formed by extreme rays $\rbar^i$ and $\rbar^j$ ($i,j \in N$). The set $\Msetone$ is composed of indices $i \in \D$ corresponding to extreme rays of $\PB$ that exhibit the following property: for every $j \in N \setminus \D$, the cone $\F_{ij}$ contains a nontrivial element of $\recc(\genset_{\D}^{\Q})$ (i.e., anything outside of $\lcinterval{0}{+\infty}{i}$).

\begin{proposition}\label{prop:epshat-recc}
	Let $(i,j) \in \Msetone \times (N \setminus \D)$. For any $\epshatone \in [0, \epshatone_{ij})$, $\enter_i \rbar^i + \epshatone \rbar^j \in \recc(\genset_{\D}^\Q)$.
\end{proposition}
\begin{proof}
	If $\gamma_{ij} = +\infty$, then $\rbar^j \in \recc(\genset_{\D}^{\Q})$, and the point $\enter_i \rbar^i + \gamma \rbar^j$ lies in $\F_{ij} \subseteq \recc(\genset_{\D}^{\Q})$. Assume $\gamma_{ij} < +\infty$. The point $\enter_i \rbar^i + \epshatone \rbar^j$ is a convex combination of $\enter_i \rbar^i + \epshatone_{ij} \rbar^j$ and $\enter_i \rbar^i$, both of which lie in $\recc(\genset_{\D}^{\Q})$:
	\begin{align*}
		\enter_i \rbar^i + \epshatone \rbar^j&= \frac{\epshatone}{\epshatone_{ij}} ( \enter_i \rbar^i + \epshatone_{ij} \rbar^j ) + \bigg(1 - \frac{\epshatone}{\epshatone_{ij}} \bigg) \enter_i \rbar^i \in \recc(\genset_{\D}^{\Q}). \tag*{\qedhere}
	\end{align*}
\end{proof}

For $j \in N \setminus \D$ and $\Sset \subseteq \Msetone$, we define $\epsone_j(\Sset)$ to be
\begin{align*}
	\epsone_j(\Sset) &=
	\begin{cases}
		\min_{i \in \Sset} \epshatone_{ij} & \textrm{ if } \Sset \neq \emptyset \\
		+\infty & \textrm{ otherwise}.
	\end{cases}
\end{align*}
The parameter $\epsone_j(\Sset)$ also depends on $\D$. We again omit this dependence for notational simplicity.

Theorem~\ref{thm:tq-cuts} presents a family of valid inequalities for $\PB \setminus \genset_{\D}^{\Q}$.
\begin{theorem}\label{thm:tq-cuts}
	Let $\Sset \subseteq \Msetone$. The inequality
	\begin{align}
		\sum\limits_{j \in \Sset} \frac{x_j}{\enter_j} - \mash{\sum\limits_{j \in N \setminus \D}} \ \frac{x_j}{\epsone_j(\Sset)} &\leq 1 \label{eq:tq-inequality}
	\end{align}
	is valid for $\PB \setminus \genset_{\D}^{\Q}$.
\end{theorem}
Prior to proving Theorem~\ref{thm:tq-cuts}, we use Farkas' lemma to derive a result on the existence of a solution to a particular family of linear equations.
\begin{lemma}\label{lem:farkas-cons-2}
	Let $\indone, \indtwo$ be two finite index sets. Let $\aparam \in \R^{|\indone|}_{+}$ and $\cparam \in \R^{|\indtwo|}_{+}$ satisfy $\sum_{i \in \indone} \aparam_i - \sum_{j \in \indtwo} \cparam_j > 0$. Then there exists $\theta \in \R^{|\indone| \times |\indtwo|}_{+}$ such that
	\begin{equation}\label{eq:fark-sys-2}
	\begin{alignedat}{2}
	\sum\limits_{i \in \indone} \theta_{ij} &= 1  &&\forall j \in \indtwo \\
	\sum\limits_{j \in \indtwo} \cparam_j \theta_{ij} &\leq \aparam_i \qquad &&\forall i \in \indone.
	\end{alignedat}
	\end{equation}
\end{lemma}
\begin{proof}
	Let $\bhat \coloneqq \sum_{j \in \indtwo} \cparam_j$. If $\bhat = 0$, then any $\theta \in \R^{|\indone| \times |\indtwo|}_{+}$ satisfying $\sum_{i \in \indone} \theta_{ij} = 1$ for all $j \in \indtwo$ is a solution to system \eqref{eq:fark-sys-2}. Therefore, assume $\bhat > 0$.
	
	Assume for contradiction \eqref{eq:fark-sys-2} does not have a solution. By Farkas' lemma, there exist $y \in \R^{|\indtwo|}$ and $z \in \R^{|\indone|}_{+}$ such that
	\begin{subequations}\label{eq:fark-family2}
		\begin{alignat}{2}
			y_j + \cparam_j z_i &\geq 0 \qquad && \forall i \in \indone,\ j \in \indtwo \label{eq:fark2-sub1} \\
			\mash{\sum\limits_{j \in \indtwo}} y_j + \mash{\sum\limits_{i \in \indone}} \aparam_i z_i &< 0. && \label{eq:fark2-sub2}
		\end{alignat}
	\end{subequations}
	We multiply \eqref{eq:fark2-sub1} by $\aparam_i/\bhat$ to obtain 
	\begin{align*}
		\frac{\aparam_i}{\bhat} y_j + \frac{\cparam_j}{\bhat} \aparam_i z_i &\geq 0 \qquad \forall i \in \indone,\ j \in \indtwo.
	\end{align*}
	Summing this expression over $i \in \indone$ and $j \in \indtwo$ produces the inequality
	\begin{align}
		\frac{\textstyle\sum_{i \in \indone} \aparam_i}{\bhat} \mash{\sum\limits_{j \in \indtwo}} y_j + \mash{\sum\limits_{i \in \indone}} \aparam_i z_i &\geq 0. \label{eq:farkas-proof-reduction}
	\end{align}
	By assumption, $\sum_{i \in \indone} \aparam_i - \bhat > 0$, which implies $\sum_{i \in \indone} \aparam_i / \bhat > 1$. Combining \eqref{eq:fark2-sub2} with the inequality $\sum_{i \in \indone} \aparam_i z_i \geq 0$, we conclude that $\sum_{j \in \indtwo} y_j < 0$. Thus, \eqref{eq:farkas-proof-reduction} implies
	\begin{align}
		\mash{\sum_{j \in \indtwo}} y_j + \mash{\sum_{i \in \indone}} \aparam_i z_i \geq 0 \label{eq:farkas-contradiction}
	\end{align}
	Inequality \eqref{eq:farkas-contradiction} contradicts \eqref{eq:fark2-sub2}. Therefore, \eqref{eq:fark-sys-2} has a solution.	
\end{proof}

\begin{proof}[Proof of Theorem~\ref{thm:tq-cuts}.]
	The statement is trivially true if $\Sset = \emptyset$. Therefore, assume $\Sset \neq \emptyset$. For ease of notation, let $\epsone_j \coloneqq \epsone_j(\Sset)$ for $j \in N \setminus \D$. Let $\helper \coloneqq \{j \in N \setminus \D \colon \epsone_j < +\infty\}$ and  $\helperc \coloneqq \{j \in N \setminus \D \colon \epsone_j = +\infty\}$. Note $j \in \helperc$ if and only if $\rbar^j \in \recc(\genset_{\D}^{\Q})$. Let $\xhat \in \PB$ satisfy $\sum_{j \in \Sset} \xhat_j / \enter_j - \sum_{j \in \helper} \xhat_j / \epsone_j > 1$. We show $\xhat \in \genset_{\D}^{\Q}$.
	
	We apply Lemma~\ref{lem:farkas-cons-2} with $\indone \coloneqq \Sset$, $\indtwo \coloneqq \helper$, $\aparam_i \coloneqq \xhat_i / \enter_i$ for $i \in \Sset$, and $\cparam_j \coloneqq \xhat_j / \epsone_j$ for $j \in \helper$. Thus, there exists $\theta \in \R^{|\Sset| \times |\helper|}_{+}$ satisfying
	\begin{subequations}\label{eq:fark2}
		\begin{alignat}{2}
			\sum\limits_{i \in \Sset} \theta_{ij} &= 1  &&\forall j \in \helper \label{eq:fark2-sum1} \\
			\mash{\sum\limits_{j \in \helper}} \ \frac{\xhat_j}{\epsone_j} \theta_{ij} 	&\leq \frac{\xhat_i}{\enter_i} \qquad &&\forall i \in \Sset. \label{eq:fark2-theta}
		\end{alignat}
	\end{subequations}
	By Proposition~\ref{prop:epshat-recc}, $q^{ij} \coloneqq \enter_i \rbar^i + \epsone_j \rbar^j \in \recc(\genset_{\D}^{\Q})$ for all $i \in \Sset$ and $j \in \helper$, because $0 < \epsone_j \leq \epshatone_{ij}$. Consequently,
	\begin{align}
		\rbar^j &= \frac{1}{\epsone_j} q^{ij} - \frac{1}{\epsone_j} \enter_i \rbar^i \qquad \forall i \in \Sset,\ j \in \helper. \label{eq:tq-rays-intermediate}
	\end{align}
	We use \eqref{eq:tq-rays-intermediate} and $\theta$ from \eqref{eq:fark2} to rewrite $\rbar^j$, $j \in \helper$:
	\begin{align}
		\rbar^j &= \sum\limits_{i \in \Sset} \theta_{ij} \bigg( \frac{1}{\epsone_j} q^{ij} - \frac{1}{\epsone_j} \enter_i \rbar^i \bigg). \label{eq:tq-rays-final}
	\end{align}
	We have $N = \D \cup \helper \cup \helperc$. Substituting \eqref{eq:tq-rays-final} into the definition of $\xhat$, we have:
	\begin{align}
		\xhat &= \xbasis + \sum\limits_{i \in \Sset} \xhat_i \rbar^i + \mash{\sum\limits_{j \in \D \setminus \Sset}} \xhat_j \rbar^j + \sum\limits_{j \in \helper} \xhat_j \rbar^j + \sum\limits_{j \in \helperc} \xhat_j \rbar^j \nonumber \\
		&= \xbasis + \sum\limits_{i \in \Sset} \xhat_i \rbar^i + \ \mash{\sum\limits_{j \in \helper}} \ \sum\limits_{i \in \Sset} \xhat_j \theta_{ij} \left( \frac{1}{\epsone_j} q^{ij} - \frac{1}{\epsone_j} \enter_i \rbar^i \right) + \mash{\sum\limits_{\quad j \in \helperc \cup (\D \setminus \Sset)}} \xhat_j \rbar^j \nonumber \\
		&= \xbasis + \sum\limits_{i \in \Sset} \bigg( \frac{\xhat_i}{\enter_i} - \mash{\sum\limits_{j \in \helper}} \theta_{ij} \frac{\xhat_j}{\epsone_j} \bigg) \enter_i \rbar^i + \sum\limits_{i \in \Sset} \ \mash{\sum\limits_{\ j \in \helper}} \theta_{ij} \frac{\xhat_j}{\epsone_j} q^{ij} + \mash{\sum\limits_{\quad j \in \helperc \cup (\D \setminus \Sset)}} \xhat_j \rbar^j. \label{eq:new-xhat}
	\end{align}
	By \eqref{eq:fark2-sum1}, the sum of the weights on the terms $\enter_i \rbar^i$, $i \in \Sset$ in \eqref{eq:new-xhat} are greater than $1$:
	\begin{align}
		\sum_{i \in \Sset} \bigg( \frac{\xhat_i}{\enter_i} - \mash{\sum_{j \in \helper}} \theta_{ij} \frac{\xhat_j}{\epsone_j} \bigg) = \sum_{i \in \Sset} \frac{\xhat_i}{\enter_i} - \mash{\sum_{j \in \helper}} \frac{\xhat_j}{\epsone_j} &> 1. \label{eq:tq:sumgeq1}
	\end{align}
	By \eqref{eq:fark2-theta}, each individual coefficient on $\enter_i \rbar^i$, $i \in \Sset$ in \eqref{eq:new-xhat} is nonnnegative. Together with \eqref{eq:tq:sumgeq1}, we have
	\begin{align}
		\xbasis + \sum\limits_{i \in \Sset} \bigg( \frac{\xhat_i}{\enter_i} - \mash{\sum\limits_{j \in \helper}} \theta_{ij} \frac{\xhat_j}{\epsone_j} \bigg) \enter_i \rbar^i \in \{\xbasis\} + \conv\bigg( \bigcup_{i \in \Sset} \interval{\enter_i}{+\infty}{i} \bigg). \label{eq:xhat-part1}
	\end{align}
	Continuing from \eqref{eq:xhat-part1}, $\{\xbasis\} + \conv ( \cup_{i \in \Sset} \interval{\enter_i}{+\infty}{i} ) \subseteq \genset_{\D}$, because $\Sset \subseteq \D$. Furthermore, because recession cone membership is preserved under addition,
	\begin{align}
		\sum\limits_{i \in \Sset} \ \mash{\sum\limits_{\ j \in \helper}} \theta_{ij} \frac{\xhat_j}{\epsone_j} q^{ij} + \mash{\sum\limits_{j \in \D \setminus \Sset}} \xhat_j \rbar^j + \sum\limits_{j \in \helperc} \xhat_j \rbar^j \in \recc(\genset_{\D}^{\Q}). \label{eq:xhat-part2}
	\end{align}
	This holds because $q^{ij} \in \recc(\genset_{\D}^{\Q})$ for all $i \in \Sset$ and $j \in \helper$, $\rbar^j \in \recc(\genset_{\D}^{\Q})$ for all $j \in \D$ by \eqref{eq:genset-def}, and $\rbar^j \in \recc(\genset_{\D}^{\Q})$ for all $j \in \helperc$. By \eqref{eq:xhat-part1} and \eqref{eq:xhat-part2}, we have $\xhat \in \genset_{\D} + \recc(\Q) = \genset_{\D}^{\Q}$.
\end{proof}
As a result of Theorem~\ref{thm:tq-cuts} and Proposition~\ref{prop:epshat-recc}, for any $\Sset \subseteq \Msetone$ and $\epshatone \in \R^{|N \setminus \D|}_{+}$ satisfying $\epshatone_j \in (0,\epsone_j(\Sset))$ for all $j \in N \setminus \D$, the inequality
\begin{align*}
	\sum_{j \in \Sset} \frac{x_j}{\enter_j} - \mash{\sum_{j \in N \setminus \D}} \ \frac{x_j}{\epshatone_j} \leq 1
\end{align*}
is valid for $\PB \setminus \genset_{\D}^{\Q}$. This is useful when calculating $\epsone_j(\Sset)$ approximately (e.g., via binary search), as it is sufficient to calculate a positive lower bound on $\epsone_j(\Sset)$. However, our choice of $\epsone_j$ yields a stronger inequality than inequalities corresponding to smaller choices of $\epshatone$.

\begin{contexample}
	Consider $\D = \N1$. Figures \ref{fig:pb-tq-4} and~\ref{fig:pb-tq-5} provide a graphical representation of Theorem~\ref{thm:tq-cuts} applied to Example~\ref{ex:tq}. The selection of $\epsone_1(\Sset)$ is shown in Figure~\ref{fig:pb-tq-4}, where $\Sset = \{2\}$. The vector $\epsone_1(\Sset) \rbar^1 + \enter_2 \rbar^2$ lies on the boundary of $\recc(\genset_{\D}^{\Q}) \cap \F_{12}$. We note $\recc(\genset_{\D}^{\Q}) = \recc(\Q)$. Figure~\ref{fig:pb-tq-5} shows the valid inequality of Theorem~\ref{thm:tq-cuts} using this selection of $\Sset$.
	
	\begin{figure}
		\begin{subfigure}[t]{0.44\textwidth}
			\vspace*{0pt}
			\centering
			\includegraphics[width=46mm]{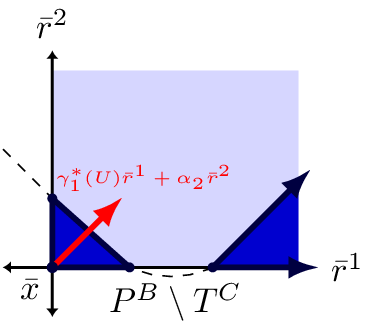}
			\caption{The maximal selection of $\epsone_1(\Sset)$, where $\Sset = \{2\}$. The vector $\epsone_1(\Sset) \rbar^1 + \enter_2 \rbar^2$ lies in $\recc(\Q)$. The darkened region is $\PB \setminus \T^\Q$, the derivation of which is shown in Figure~\ref{fig:pb-tq}.}
			\label{fig:pb-tq-4}
		\end{subfigure}%
		\begin{subfigure}[t]{0.05\textwidth}
			\quad
		\end{subfigure}%
		\begin{subfigure}[t]{0.44\textwidth}
			\vspace*{0pt}
			\centering
			\includegraphics[width=46mm]{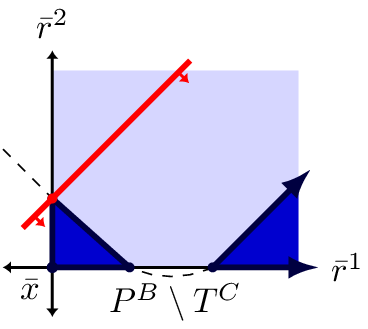}
			\caption{Theorem~\ref{thm:tq-cuts}'s inequality, $x_2/\enter_2 - x_1 / \epsone_1(\Sset) \leq 1$, is valid for $\PB \setminus \T^\Q$. The corresponding hyperplane $\{x \in \R^2 \colon x_2/\enter_2 - x_1 / \epsone_1(\Sset) = 1\}$ contains the point $\xbasis + \enter_2 \rbar^2$. The vector $\epsone_1(\Sset) \rbar^1 + \enter_2 \rbar^2$ is a recession direction of this hyperplane.}
			\label{fig:pb-tq-5}
		\end{subfigure}
		\caption{The valid inequality of Theorem~\ref{thm:tq-cuts} applied to Example~\ref{ex:tq}.}
	\end{figure}
\end{contexample}

We next consider the problem of selecting a subset of $\Msetone$ that yields the most violated inequality of the form \eqref{eq:tq-inequality} to cut off a candidate solution $\xhat \in \PB$. That is, we are interested in the separation problem
\begin{align}
	\max_{\Sset \subseteq \Msetone} \sum\limits_{j \in \Sset} \frac{\xhat_j}{\enter_j} - \mash{\sum\limits_{j \in N \setminus \D}} \ \frac{\xhat_j}{\epsone_j(\Sset)}. \label{eq:tq-separation}
\end{align}
For each $j \in N \setminus \D$, we define the function $\funcone_j \colon 2^{\Msetone} \rightarrow \R$ to be $\funcone_j(\Sset) = -\xhat_j / \min_{i \in \Sset} \epshatone_{ij}$ if $\Sset \neq \emptyset$, and $0$ otherwise. The value $\funcone_j(\Sset)$ is the contribution of index $j \in N \setminus \D$ to the objective function \eqref{eq:tq-separation} for a given $\Sset$.
\begin{proposition}\label{prop:tq-separation}
	The maximization problem \eqref{eq:tq-separation} is a supermodular maximization problem.
\end{proposition}
\begin{proof}
	The separation problem \eqref{eq:tq-separation} can be equivalently written as
	\begin{align}
		\max_{\Sset \subseteq \Msetone} \sum\limits_{j \in \Sset} \frac{\xhat_j}{\enter_j} + \mash{\sum\limits_{j \in N \setminus \D}} \funcone_j(\Sset). \label{eq:tq-separation-rewrite}
	\end{align}
	From standard properties of the $\min$ operator, the objective function of \eqref{eq:tq-separation-rewrite} is the sum of supermodular (and modular) functions.
\end{proof}
By Proposition~\ref{prop:tq-separation}, the separation problem \eqref{eq:tq-separation} can be solved in strongly polynomial time (e.g., \citet{grotschel1981}, \citet{grotschel2012}, \citet{orlin2009}).

We next propose an extended formulation for the relaxation of $\PB \setminus \genset_{\D}^{\Q}$ defined by inequality \eqref{eq:tq-inequality} for all $\Sset \subseteq \Msetone$:
\begin{align*}
	\tqrelax_{\D} \coloneqq \bigg\lbrace x \in \R^{|N|}_{+} \colon \sum\limits_{j \in \Sset} \frac{x_j}{\enter_j} - \mash{\sum\limits_{j \in N \setminus \D}} \ \frac{x_j}{\epsone_j(\Sset)} \leq 1 \ \forall \Sset \subseteq \Msetone \bigg\rbrace.
\end{align*}
Let $\Msetone = \{1, \ldots, \tqcardone\}$ and $\tqcardtwo \coloneqq |N \setminus \D|$. For each $j \in N \setminus \D$, let $\pi_j(1), \pi_j(2), \ldots, \pi_j(\tqcardone)$ be ordered such that $\epshatone_{\pi_j(1),j} \leq \epshatone_{\pi_j(2),j} \leq \ldots \leq \epshatone_{\pi_j(\tqcardone),j}$. Similarly, for any $i \in \Msetone$, let $\ellone_j(i)$ satisfy $\pi_j(\ellone_j(i)) = i$. For all $j \in N \setminus \D$, let $\epshatone_{0j} \coloneqq +\infty$, $\theta_{0j} \coloneqq 0$, $v_{0j} \coloneqq 0$, $v_{\tqcardone+1,j} \coloneqq 0$, $\pi_j(0) \coloneqq 0$, and $\pi_j(\tqcardone+1) \coloneqq 0$. We define $\tqextended_{\D}$ to be the set of $(x,\theta,v,\lambda) \in \R_{+}^{|N|} \times \R_{+}^{\tqcardone \times \tqcardtwo} \times \R_{+}^{\tqcardone \times \tqcardtwo} \times \R^{\tqcardtwo}$ such that

\begin{equation*}\label{eq:sepip-dual}
	\begin{alignedat}{2}
		\sum\limits_{i \in \Msetone} \enspace \mash{\sum\limits_{j \in N \setminus \D}} \theta_{ij} + \mash{\sum\limits_{j \in N \setminus \D}} \lambda_j &\leq 1 \\
		\theta_{ij} + v_{ij} - v_{i+1,j} + \bigg( \frac{1}{\epshatone_{\pi_j(i+1),j}} - \frac{1}{\epshatone_{\pi_j(i),j}} \bigg) x_j &\geq 0 \quad && \forall i = 0, \ldots, \tqcardone,\ j \in N \setminus \D \\
		\mash{\sum\limits_{j \in N \setminus \D}} \theta_{\ellone_j(i),j} - \frac{1}{\enter_i}x_i &\geq 0 \quad && \forall i = 1, \ldots, \tqcardone.
	\end{alignedat}
\end{equation*}
Theorem~\ref{thm:extendedform1} establishes the relationship between $\tqextended_{\D}$ and the relaxation $\tqrelax_{\D}$.
\begin{theorem}\label{thm:extendedform1}
	The polyhedron $\tqextended_{\D}$ is an extended formulation of $\tqrelax_{\D}$:
	\begin{align*}
		\proj_x(\tqextended_{\D}) = \tqrelax_{\D}.
	\end{align*}
\end{theorem}
\begin{proof}
	We first argue that the following linear program solves the separation problem \eqref{eq:tq-separation} for a fixed $\xhat \in \PB$:
	\begin{subequations}\label{eq:sepip}
		\begin{alignat}{4}
			\max_{y,z}\ && \sum\limits_{i \in \Msetone} \frac{\xhat_i}{\enter_i} z_i - &\mash{\sum\limits_{j \in N \setminus \D}} \enspace\ \sum\limits_{i = 1}^{\tqcardone} && \frac{\xhat_j}{\epshatone_{\pi_j(i),j}} (y_{i-1,j} - y_{ij}) && \label{eq:sepip-obj} \\
			\textrm{s.t.}\ && y_{0j} &= 1 && \forall j \in N \setminus \D && (\lambda_j) \label{eq:sepip-1} \\
			&& y_{ij} + z_{\pi_j(i)} &\leq 1 && \forall i = 1, \ldots, \tqcardone,\ j \in N \setminus \D \qquad && (\theta_{ij})\label{eq:sepip-2} \\
			&& y_{ij} - y_{i-1,j} &\leq 0 && \forall i = 1, \ldots, \tqcardone,\ j \in N \setminus \D && (v_{ij}) \label{eq:sepip-3} \\
			&& y_{ij} &\geq 0 && \forall i = 0, \ldots, \tqcardone,\ j \in N \setminus \D \label{eq:sepip-4} \\
			&& z_i &\geq 0 && \forall i \in \Msetone. \label{eq:sepip-5}
		\end{alignat}
	\end{subequations}		
	The constraint matrix of $\eqref{eq:sepip}$ is totally unimodular. To see this, we complement the $z_i$ variables with $1 - z_i$ for all $i \in \Msetone$ to obtain an equivalent problem. The resulting constraint matrix has $0,\pm 1$ entries, and each row contains no more than one $1$ and one $-1$.
	
	We show \eqref{eq:sepip} correctly models the separation problem \eqref{eq:tq-separation}. First, let $\Sset^*$ be the optimal solution of \eqref{eq:tq-separation}. We construct $(y^*,z^*)$ feasible to \eqref{eq:sepip} with objective function value equal to $\sum_{i \in \Sset^*} \xhat_i / \enter_i - \sum_{j \in N \setminus \D} \xhat_j / \epsone_j(\Sset^*)$. If $\Sset^* = \emptyset$, then the optimal objective value of the separation problem is $0$. In this case, set $y_{ij}^*=1$ for all $i = 0,1, \ldots, \tqcardone$ and $j \in N \setminus \D$, and set $z_i^*=0$ for all $i \in \Msetone$. Then $(y^*,z^*)$ is feasible to \eqref{eq:sepip} with objective value $0$. If $\Sset^* \neq \emptyset$, set $z_i^* = 1$ if $i \in \Sset^*$, and $0$ otherwise. For all $j \in N \setminus \D$, let $k_j \in \Sset^*$ be the smallest index satisfying $\pi_j(k_j) \in \arg \min_{i \in \Sset^*} \epshatone_{ij}$; that is, $k_j = \min \{k \colon \epshatone_{\pi_j(k_j),j} = \epsone_j(\Sset^*)\}$. For each $j \in N \setminus \D$, set $y^*_{ij} = 1$ for all $i = 0, 1, \ldots, k_j - 1$, and set $y^*_{ij} = 0$ for all $i = k_j, \ldots,  \tqcardone$. By construction, $(y^*,z^*)$ satisfies \eqref{eq:sepip-1}--\eqref{eq:sepip-5}. For a fixed $j \in N \setminus \D$, we have
	\begin{align*}
		\sum\limits_{i = 1}^{\tqcardone} \frac{\xhat_j}{\epshatone_{\pi_j(i),j}} (y^*_{i-1,j} - y^*_{ij}) = \frac{\xhat_j}{\epshatone_{\pi_j(k_j),j}} = \frac{\xhat_j}{\epsone_j(\Sset^*)},
	\end{align*}
	and the objective function \eqref{eq:sepip-obj} evaluates to the desired value of
	\begin{align*}
		\sum\limits_{i \in \Msetone} \frac{\xhat_i}{\enter_i} z^*_i - \mash{\sum\limits_{j \in N \setminus \D}} \enspace\ \sum\limits_{i = 1}^{\tqcardone}\frac{\xhat_j}{\epshatone_{\pi_j(i),j}} (y^*_{i-1,j} - y^*_{ij}) = \sum\limits_{i \in \Sset^*} \frac{\xhat_i}{\enter_i} - \mash{\sum\limits_{j \in N \setminus \D}} \ \frac{\xhat_j}{\epsone_j(\Sset^*)}.
	\end{align*}
	Now, let $(y^*,z^*)$ be an optimal solution to \eqref{eq:sepip}. Set $\Sset^* = \{i \in \Msetone \colon z^*_i = 1\}$. It remains to show $\sum_{i \in \Sset^*} \xhat_i / \enter_i - \sum_{j \in N \setminus \D} \xhat_j / \epsone_j(\Sset^*)$ is not less than the optimal objective value of \eqref{eq:sepip}. Recall the constraint matrix of \eqref{eq:sepip} is totally unimodular, so $(y^*,z^*)$ is $0$--$1$ valued. If $z^*_i = 0$ for all $i \in \Msetone$, then the separation problem objective evaluated at $\Sset^* = \emptyset$ is $0$ and the optimal objective value of \eqref{eq:sepip} is nonpositive, as desired. Next, assume $\sum_{i \in \Msetone} z^*_i \geq 1$. By constraints \eqref{eq:sepip-2}, given $j \in N \setminus \D$, $y^*_{ij} = 0$ for all $i \in \Msetone$. By constraints \eqref{eq:sepip-1} and \eqref{eq:sepip-3}, for each $j \in N \setminus \D$, there exists $k_j$ such that $y_{ij} = 1$ for $i = 0, \ldots, k_j - 1$ and $y_{ij} = 0$ for $i = k_j, \ldots, \tqcardone$. Then the optimal objective value of \eqref{eq:sepip} is
	\begin{align}
		\sum\limits_{i \in \Msetone} \frac{\xhat_i}{\enter_i} z^*_i - \mash{\sum\limits_{j \in N \setminus \D}} \enspace\ \sum\limits_{i = 1}^{\tqcardone}\frac{\xhat_j}{\epshatone_{\pi_j(i),j}} (y^*_{i-1,j} - y^*_{ij}) = \sum\limits_{i \in \Msetone} \frac{\xhat_i}{\enter_i} z^*_i - \mash{\sum\limits_{j \in N \setminus \D}} \ \frac{\xhat_j}{\epshatone_{\pi_j(k_j),j}}. \label{eq:sepipi-optobj}
	\end{align}
	Consider a fixed $j \in N \setminus \D$. By constraints \eqref{eq:sepip-2}, $z_{\pi_j(i)} = 0$ for $i = 1, \ldots, k_j - 1$. Then $\arg\min \{i \in \Msetone \colon z^*_i = 1\} \geq k_j$. Due to the ordering $\epshatone_{\pi_j(1),j} \leq \ldots \leq \epshatone_{\pi_j(\tqcardone),j}$, we have $\epsone_j(\Sset^*) = \min_{i \in \Sset^*} \epshatone_{ij} \geq \epshatone_{\pi_j(k_j),j}$. Therefore, the optimal objective value of the separation problem evaluated at $\Sset^*$ is at least as large as \eqref{eq:sepipi-optobj}:
	\begin{align*}
		\sum\limits_{i \in \Sset^*} \frac{\xhat_i}{\enter_i} - \mash{\sum\limits_{j \in N \setminus \D}} \ \frac{\xhat_j}{\epsone_j(\Sset^*)} 
		&\geq \sum\limits_{i \in \Msetone} \frac{\xhat_i}{\enter_i} z^*_i - \mash{\sum\limits_{j \in N \setminus \D}} \ \frac{\xhat_j}{\epshatone_{\pi_j(k_j),j}}.
	\end{align*}
	Hence, \eqref{eq:sepip} models the separation problem \eqref{eq:tq-separation} for a fixed $\xhat_j \in \PB$.
	
	We conclude by relating $\tqextended_{\D}$ to \eqref{eq:sepip}. The point $\xhat$ lies in $\tqrelax_{\D}$ if and only if the primal objective \eqref{eq:sepip-obj} does not exceed $1$. The linear program \eqref{eq:sepip} is feasible and bounded, so strong duality applies. Let $\lambda$, $\theta$, and $v$ be the linear program's dual variables, as labeled in \eqref{eq:sepip}. By strong duality, $\xhat \in \tqrelax_{\D}$ if and only if the dual of \eqref{eq:sepip} has objective value less than or equal to $1$. Because the dual of \eqref{eq:sepip} is a minimization problem, we enforce this condition with the constraint $\sum_{i \in \Msetone} \sum_{j \in N \setminus \D} \theta_{ij} + \sum_{j \in N \setminus \D} \lambda_j \leq 1$. We also replace the fixed $\xhat$ in the dual of \eqref{eq:sepip} with the nonnegative variable $x \in \R^{|N|}_{+}$. Thus, $x \in \tqrelax_{\D}$ if and only if there exists $(\theta,v,\lambda)$ satisfying the dual constraints of \eqref{eq:sepip} and the aforementioned dual objective cut. These constraints define $\tqextended_{\D}$.
\end{proof}

Within the proof of Theorem~\ref{thm:extendedform1}, we show that the linear program \eqref{eq:sepip} can be used to solve the separation problem \eqref{eq:tq-separation}. The remainder of the proof uses the separation linear program \eqref{eq:sepip} and duality theory to derive an extended formulation, a technique that was first proposed by \citet{martin1991}.

For a fixed $j \in N \setminus \D$, the constraints \eqref{eq:sepip-1}, \eqref{eq:sepip-3}, and \eqref{eq:sepip-4} form an instance of the mixing set, first studied by \citet{gunluk2001}. The proof's derivation of the extended formulation $\tqextended_{\D}$ follows results from \citet{luedtke2008} and \citet{miller2003}.

Proposition~\ref{prop:no-tq-cuts} states that if no cuts of the form \eqref{eq:tq-inequality} exist, then there exist no valid inequalities for $\clconv (\PB \setminus \Q)$ other than those defining $\PB$.
\begin{proposition}\label{prop:no-tq-cuts}
	Under Assumption~\ref{ass:recc}, if $\Msetone = \emptyset$, then $\clconv(\PB \setminus \genset_{\D}^{\Q}) = \PB$.
\end{proposition}
\begin{proof}
	It suffices to show $\{\xbasis\} + \lcinterval{0}{+\infty}{i} \subseteq \clconv(\PB \setminus \genset_{\D}^{\Q})$ for $i \in N$.
	
	We first show $\{\xbasis\} + \lcinterval{0}{+\infty}{i} \subseteq \PB \setminus \genset_{\D}^{\Q}$ for $i \in N \setminus \D$. Assume for contradiction there exists $k \in N \setminus \D$ and $\gamma \geq 0$ such that $\xbasis + \gamma \rbar^k \in \genset_{\D}^{\Q}$. By the definition of $\genset_{\D}^{\Q}$ in \eqref{eq:genset-def}, there exists $\lambda \in \R^{|D|}_{+}$, $\theta \in \R^{|\D|}_{+}$, and $q \in \recc(\Q)$ such that $\lambda_j > \enter_j$ for all $j \in \D$, $\sum_{j \in \D} \theta_j = 1$, and
	\begin{align*}
		\xbasis + \gamma \rbar^k &= \xbasis + \sum\limits_{j \in \D} \theta_j \lambda_j \rbar^j + q.
	\end{align*}
	Equivalently, we have $q = \gamma \rbar^k - \sum_{j \in \D} \theta_j \lambda_j \rbar^j$. Because the vectors $\{\rbar^j \colon j \in N\}$ are linearly independent and there exists $k \in \D$ such that $-\theta_j \lambda_j < 0$, it holds that $q \notin \cone(\{\rbar^j \colon j \in N\}) = \recc(\PB)$. This contradicts Assumption~\ref{ass:recc}, which states $\recc(\Q) \subseteq \recc(\PB)$.
	
	Now, consider $i \in \D$, $\lambda > 0$, and $\epshatone > 0$. Because $\Msetone = \emptyset$, there exists $j \in N \setminus \D$ such that $\lambda \rbar^i + \epshatone \rbar^j \notin \recc(\genset_{\D}^{\Q})$. Then for a sufficiently large $M > 1$, $\xbasis + M (\lambda \rbar^i + \epshatone \rbar^j) \notin \genset_{\D}^{\Q}$. We have that $\xbasis + \lambda \rbar^i + \epshatone \rbar^j$ is a convex combination of $\xbasis \notin \genset_{\D}^{\Q}$ and $\xbasis + M (\lambda \rbar^i + \epshatone \rbar^j) \notin \genset_{\D}^{\Q}$. Thus, $\xbasis + \lambda \rbar^i + \epshatone \rbar^j \in \conv(\PB \setminus \genset_{\D}^{\Q})$ for all $\epshatone > 0$, so $\xbasis + \lambda \rbar^i \in \clconv(\PB \setminus \genset_{\D}^{\Q})$.
\end{proof}
In the case where $\D = \N1$ ($\genset_{\D}^{\Q} = \Rset^\Q$) and $\Msetone = \emptyset$, Theorem~\ref{thm:conv-cq} and Proposition~\ref{prop:no-tq-cuts} together give us $\Msetone = \emptyset \implies \clconv(\PB \setminus \Q) = \PB$.

\subsection{A multi-term disjunction valid for \texorpdfstring{$\PB \setminus \T^\Q$}{P\carrot B \back T\carrot C}}\label{subsec:multiterm}

The set $\PB \setminus \T^\Q$ has the potential to be a tighter relaxation of $\PB \setminus \Q$ than the two-term disjunction of Theorem~\ref{thm:n1-n2}, because it considers the full structure of $\recc(\Q)$. In this section, we derive a valid disjunction for $\PB \setminus \T^\Q$ that contains $|\N2| + 1$  terms. These terms are defined by nonconvex sets, but in Section~\ref{subsec:skq-cuts-combined} we derive polyhedral relaxations of each term. Given the valid disjunction we derived for $\PB \setminus \T^\Q$, these polyhedral relaxations can be used with other inequalities defining $\P$ to obtain a union of polyhedra that contains $\P \setminus \T^\Q$. The disjunctive programming approach of Balas can be applied to construct a CGLP to find a valid inequality that separates a candidate solution from $\clconv(\P \setminus \T^\Q)$.

Let $\So^\Q$ be defined as follows:
\begin{align}\label{eq:def-s0q}
	\So^\Q &\coloneqq \{\xbasis\} + \conv \big( \textstyle\bigcup_{j \in \Nonetwo} \interval{\enter_j}{+\infty}{j} \big) + \recc(\Q).
\end{align}
We define the following sets for $k \in \N2$:
\begin{alignat}{2}
	\Sk^\Q &\coloneqq \{\xbasis\} + \conv \big( \textstyle\bigcup_{j \in \N2} \lcinterval{0}{\exit_j}{j} \big) + \rcinterval{-\infty}{0}{k} + \recc(\Q). \label{eq:def-skq}
\end{alignat}
The sets $\So^\Q$ and $\Sk^\Q$ ($k \in \N2$) are the foundation of our multi-term valid disjunction for $\PB \setminus \T^\Q$.

In Proposition~\ref{prop:rewrite-skq}, we present an equivalent construction of $\Sk^\Q$ ($k \in \N2$).
\begin{proposition}\label{prop:rewrite-skq}
	For $k \in \N2$, $\Sk^\Q$ can be written as
	\begin{align}
		\Sk^\Q &= \{\xbasis\} + \conv \big( \textstyle\bigcup_{j \in \Nonetwo} \interval{\enter_j}{\exit_j}{j} \big) + \rcinterval{-\infty}{0}{k} + \recc(\Q). \label{eq:def-skq-alt}
	\end{align}
\end{proposition}
\begin{proof}
	For $A_1,A_2 \subseteq \R^n$, it holds that $\conv(A_1 + A_2) = \conv(A_1) + \conv(A_2)$. For $B \subseteq \R^n$, we also have $(A_1 \cup A_2) + B = (A_1 + B) \cup (A_2 + B)$. This gives us:
	\begin{align}
		\Sk^\Q &= \{\xbasis\} + \conv \big( \textstyle\bigcup_{j \in \N2} \lcinterval{0}{\exit_j}{j} \big) + \rcinterval{-\infty}{0}{k} + \recc(\Q) \nonumber \\
		&= \{\xbasis\} + \conv \big( \textstyle\bigcup_{j \in \N2} ( \lcinterval{0}{\exit_j}{j} + \rcinterval{-\infty}{0}{k} + \recc(\Q) ) \big). \label{eq:skq-stepone}
	\end{align}
	Observe that $\interval{\enter_j}{\exit_j}{j} \subseteq \{0\} + \recc(\Q)$ for all $j \in \N1$. This allows us to rewrite $\Sk^\Q$ from \eqref{eq:skq-stepone} as
	\begin{align}
		\Sk^\Q &= \{\xbasis\} + \conv \big( \textstyle\bigcup_{j \in \Nonetwo} ( \lcinterval{0}{\exit_j}{j} + \rcinterval{-\infty}{0}{k} + \recc(\Q) ) \big). \label{eq:skq-prep-replace}
	\end{align}
	Finally, note that $0 \in \interval{\enter_k}{\exit_k}{k} + \rcinterval{-\infty}{0}{k}$. Hence, for all $j \in \Nonetwo$, we can replace $\lcinterval{0}{\exit_j}{j}$ in the convex hull operator of \eqref{eq:skq-prep-replace} with $\interval{\enter_j}{\exit_j}{j}$:
	\begin{align*}
		\Sk^\Q &= \{\xbasis\} + \conv \big( \textstyle\bigcup_{j \in \Nonetwo} ( \interval{\enter_j}{\exit_j}{j} + \rcinterval{-\infty}{0}{k} + \recc(\Q) ) \big) \\
		&= \{\xbasis\} + \conv \big( \textstyle\bigcup_{j \in \Nonetwo} \interval{\enter_j}{\exit_j}{j} \big) + \rcinterval{-\infty}{0}{k} + \recc(\Q).  \tag*{\qedhere}
	\end{align*}
\end{proof}

Theorem~\ref{thm:disjunction} presents a disjunctive representation of $\PB \setminus \T^\Q$. Throughout, let $\Ncup \coloneqq \N2 \cup \{0\}$.
\begin{theorem}\label{thm:disjunction}
	It holds that
	\begin{align}
		\PB \setminus \T^\Q = \mash{\bigcup_{k \in \Ncup}} \ (\PB \setminus \Sk^\Q). \label{eq:disjunction}
	\end{align}
\end{theorem}
Before proving Theorem~\ref{thm:disjunction}, we prove a consequence of Farkas' lemma \cite{farkas1902}.
\begin{lemma}\label{lem:farkas-cons-1}
	Let $\aparam, \cparam \in \R^{\lparam}_{+}$, where $\sum_{i = 1}^{\lparam} \aparam_i > 0$. There exists $\theta \in \R^{\lparam + 1}_{+}$ such that
	\begin{align}\label{eq:fark-sys-1}
		\begin{split}
			\sum\limits_{i = 0}^{\lparam} \theta_i &= 1 \\
			\aparam_i \theta_0  - \cparam_i \theta_i &= 0  \qquad i = 1, \ldots, \lparam.
		\end{split}
	\end{align}
\end{lemma}
\begin{proof}
	By Farkas' lemma, either system \eqref{eq:fark-sys-1} has a solution, or there exists $y \in \R^{\lparam + 1}$ such that
	\begin{subequations}\label{eq:fark-family1}
		\begin{align}
			y_0 + \sum\limits_{i = 1}^{\lparam} \aparam_i y_i &\geq 0 \label{eq:fark-sub1} \\
			y_0 - \cparam_i y_i &\geq 0  \qquad i = 1, \ldots, \lparam \label{eq:fark-sub2} \\
			y_0 &< 0. \label{eq:fark-sub3}
		\end{align}
	\end{subequations}
	Assume for contradiction there exists a $y$ satisfying \eqref{eq:fark-family1}. The nonnegativity of $\cparam$, \eqref{eq:fark-sub2}, and \eqref{eq:fark-sub3} imply $y_i < 0$ for all $i = 1, \ldots, \lparam$. The vector $\aparam$ is nonnegative and by assumption sums to a strictly positive value. We conclude $y_0 + \sum_{i = 1}^{\lparam} \aparam_i y_i < 0$, contradicting \eqref{eq:fark-sub1}.
\end{proof}
\begin{proof}[Proof of Theorem~\ref{thm:disjunction}.]
	It suffices to show $\T^\Q = \bigcap_{k \in \Ncup} \Sk^\Q$. If $\N2 = \emptyset$, we have $\T^\Q = \So^\Q$ by \eqref{eq:def-tq} and \eqref{eq:def-s0q}, and the result holds. Therefore, assume $\N2 \neq \emptyset$. By construction, $\T^\Q \subseteq \Sk^\Q$ for all $k \in \Ncup$, implying $\T^\Q \subseteq \cap_{k \in \Ncup} \Sk^\Q$.
	
	Let $\xhat \in \cap_{k \in \Ncup} \Sk^\Q$. By $\xhat$'s membership in $\So^\Q$, there exist $\lambda^0 \in \R^{|\Nonetwo|}_{++}$, $\mu \in \R^{|\N2|}_{++}$, $\delta^0 \in \R_{+}^{|\Nonetwo|}$, and $q^0 \in \recc(\Q)$ such that $\lambda^0_j \in (\enter_j, \exit_j)$ for all $j \in \Nonetwo$, $\sum_{j \in \Nonetwo} \delta^0_j = 1$, and
	\begin{align}\label{eq:proof-xhat-s0}
		\xhat &= \xbasis + \sum\limits_{j \in \N1} \delta^0_j \lambda^0_j \rbar^j + \sum\limits_{j \in \N2} \delta^0_j (\lambda^0_j + \mu_j) \rbar^j + q^0.
	\end{align}
	If $\delta^0_j = 0$ for all $j \in \N2$, then $\xhat \in \{\xbasis\} + \conv( \cup_{j \in \N1} \interval{\enter_j}{\exit_j}{j} ) + \recc(\Q) \subseteq \T^\Q$ by \eqref{eq:proof-xhat-s0} and we have nothing left to prove. We therefore assume $\sum_{j \in \N2} \delta^0_j > 0$. From \eqref{eq:def-skq-alt}, $\xhat \in \cap_{k \in \N2} \Sk^\Q$ implies that for all $k \in \N2$, there exist $\lambda^k \in \R^{|\Nonetwo|}_{++}$, $\eta_k \in \R_{+}$, $\delta^k \in \R_{+}^{|\Nonetwo|}$, and $q^k \in \recc(\Q)$ such that $\lambda^k_j \in (\enter_j, \exit_j)$ for all $j \in \Nonetwo$, $\sum_{j \in \Nonetwo} \delta^k_j = 1$, and
	\begin{align}\label{eq:proof-xhat-sk}
		\xhat &= \xbasis + \mash{\sum\limits_{j \in \Nonetwo}} \ \delta^k_j \lambda^k_j \rbar^j - \eta_k \rbar^k + q^k.
	\end{align}
	Because $\mu \in \R^{|\N2|}_{++}$ and $\sum_{j \in \N2} \delta^0_j > 0$, it holds that $\sum_{j \in \N2} \delta^0_j \mu_j > 0$. We apply Lemma~\ref{lem:farkas-cons-1} with $\lparam \coloneqq |\N2|$, $\aparam_j \coloneqq \delta_j^0 \mu_j$ for $j \in \N2$, and $\cparam_j \coloneqq \eta_j$ for $j \in \N2$. Then there exists $\theta \in \R^{|\N2| + 1}_{+}$ such that $\sum_{j \in \Ncup} \theta_j = 1$ and $\theta_0 \delta^0_k \mu_k = \theta_k \eta_k$ for all $k \in \N2$. We use this $\theta$ as convex combination multipliers on \eqref{eq:proof-xhat-s0} and \eqref{eq:proof-xhat-sk} to rewrite $\xhat$ as
	\begin{align}
		\xhat &= \xbasis + \mash{\sum\limits_{k \in \Ncup }} \ \mash{\sum\limits_{\ j \in \Nonetwo}} \theta_k \delta^k_j \lambda^k_j \rbar^j + \mash{\sum\limits_{k \in \Ncup}} \theta_k q^k. \label{eq:xhat-conv-rep}
	\end{align}
	For every $j \in \Nonetwo$ and $k \in \Ncup$, $\lambda^k_j \rbar^j \in \interval{\enter_j}{\exit_j}{j}$. The coefficients on the terms $\lambda^k_j \rbar^j$ ($j \in \Nonetwo$, $k \in \Ncup$) in \eqref{eq:xhat-conv-rep} are nonnegative and sum to one:
	\begin{align*}
		\mash{\sum\limits_{k \in \Ncup }} \ \mash{\sum\limits_{\ j \in \Nonetwo}} \theta_k \delta^k_j &= \mash{\sum\limits_{k \in \Ncup }} \theta_k \mash{\sum\limits_{j \in \Nonetwo}} \delta^k_j = 1.
	\end{align*}
	It follows that $\xbasis + \sum_{k \in \Ncup} \sum_{j \in \Nonetwo} \theta_k \delta^k_j \lambda^k_j \rbar^j \in \T$. Lastly, we have $\sum_{k \in \Ncup} \theta_k q^k \in \recc(\Q)$. Thus, by \eqref{eq:xhat-conv-rep}, $\xhat \in \T^\Q$.
\end{proof}

The multi-term disjunction \eqref{eq:disjunction} is a generalization of the two-term disjunction of Theorem~\ref{thm:n1-n2}. Recall that this two-term disjunction does not account for the recession structure of $\Q$ beyond the property that $\rbar^j \in \recc(\Q)$ for all $j \in \N1$ and the assumption $\rbar^j \in \N0$ for all $j \in \N0$. If $\Q$ is bounded and Assumption~\ref{ass:1} holds, it can be shown that the multi-term disjunction reduces to the simple two-term disjunction of Theorem~\ref{thm:n1-n2}. In particular, we have
\begin{align*}
	\PB \setminus \So^\Q &= \{x \in \PB \colon \textstyle\sum_{j \in N} x_j / \enter_j \leq 1\} \\
	\PB \setminus \Sk^\Q &= \{x \in \PB \colon \textstyle\sum_{j \in N} x_j / \exit_j \geq 1\} \quad \forall k \in \N2.
\end{align*}
In the remainder of this paper, we derive polyhedral relaxations for each of the terms in the disjunction \eqref{eq:disjunction}. Given a polyhedral relaxation of each disjunctive term, we can obtain valid inequalities for $\P \setminus \Q$ using a disjunctive approach analogous to the method outlined in Remark~\ref{remark:0}.

\begin{remark}
	The multi-term disjunction \eqref{eq:disjunction} for $\P \setminus \Q$ can be extended to the case $\xbasis \in \Q$. Specifically, if $\enter_j = 0$ for all $j \in N$, the set $\So^\Q$ defined in \eqref{eq:def-s0q} contains every point in $\PB$ except for $\xbasis$. Because $\xbasis \in \Q$, we know that $\PB \setminus \So^\Q$ (one of the terms of the disjunction \eqref{eq:disjunction}) is empty.
\end{remark}
\addtocounter{example}{-1}
\begin{contexample}
	Using the two-term disjunction from Section~\ref{sec:cuts-1}, we were unable to derive meaningful cuts for $\P \setminus \Q$ from Example~\ref{ex:useless}. In contrast, Theorem~\ref{thm:disjunction} provides a disjunction for $\P \setminus \Q$. A graphical representation of the relationship $\T^\Q = \bigcap_{k \in \Ncup} \Sk^\Q$ for this example is shown in Figures \ref{fig:tq-skq-2}--\ref{fig:tq-skq-1}. In this example, $|\N0| = |\N2| = 1$. The disjunction of Theorem~\ref{thm:disjunction} can be seen in Figure~\ref{fig:tq-dis-1}.
	\begin{figure}
		\begin{subfigure}[t]{0.3\textwidth}
			\vspace*{0pt}
			\centering
			\includegraphics[width=46mm]{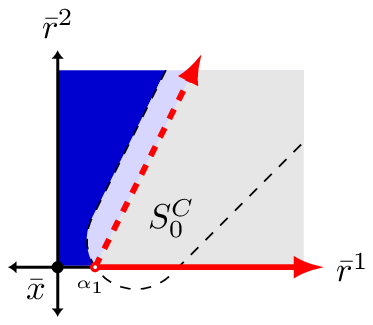}
			\caption{The set $\So^\Q$ for Example~\ref{ex:useless}. This set is one of the terms of the disjunction \eqref{eq:disjunction}.}
			\label{fig:tq-skq-2}
		\end{subfigure}%
		\begin{subfigure}[t]{0.05\textwidth}
			\quad
		\end{subfigure}%
		\begin{subfigure}[t]{0.29\textwidth}
			\vspace*{0pt}
			\centering
			\includegraphics[width=46mm]{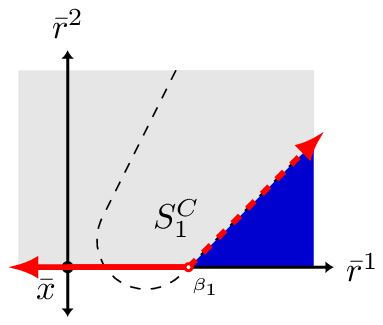}
			\caption{The set $\Sbare_1^\Q$ for Example~\ref{ex:useless}. This set is one of the terms of the disjunction \eqref{eq:disjunction}.}
			\label{fig:tq-skq-3}
		\end{subfigure}%
		\begin{subfigure}[t]{0.05\textwidth}
			\quad
		\end{subfigure}%
		\begin{subfigure}[t]{0.3\textwidth}
			\vspace*{0pt}
			\centering
			\includegraphics[width=46mm]{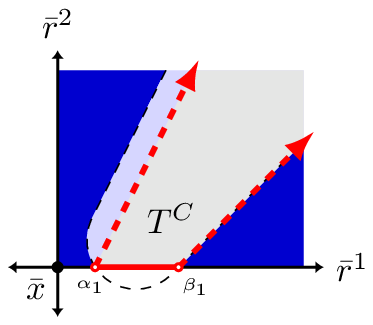}
			\caption{The set $\T^\Q$ is the intersection of $\So^\Q$ and $\Sbare_1^\Q$.}
			\label{fig:tq-skq-1}
		\end{subfigure}
		\caption{Construction of $\T^\Q$ for Example~\ref{ex:useless}.}
	\end{figure}
	
	\begin{figure}[ht]
		\centering
		\includegraphics[width=46mm]{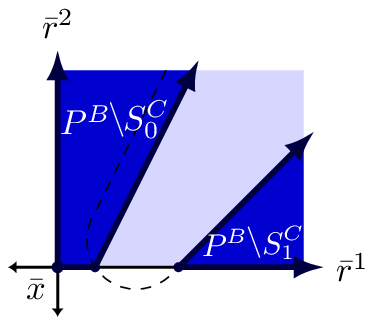}
		\caption{By Theorem~\ref{thm:disjunction}, $(\PB \setminus \So^\Q) \cup (\PB \setminus \Sbare_1^\Q)$ is a relaxation of $\PB \setminus \Q$. These sets are shown for Example~\ref{ex:useless}.}
		\label{fig:tq-dis-1}
	\end{figure}
\end{contexample}
\addtocounter{example}{1}
Based on the disjunction \eqref{eq:disjunction}, the inequalities \eqref{eq:tq-inequality}, which are valid for $\PB \setminus \T^\Q$, are also valid for $\PB \setminus \Sk^\Q$ for all $k \in \Ncup$.

The sets $\PB \setminus \Sk^\Q$, $k \in \Ncup$ are nonconvex in general. In Section~\ref{subsec:skq-cuts-combined}, we derive polyhedral relaxations of these sets. Together, these relaxations form $|\N2| + 1$ polyhedra whose union contains the feasible region $\P \setminus \Q$.

\subsection{Polyhedral relaxation of \texorpdfstring{$\PB \setminus \Sk^\Q$, $k \in \Ncup$}{P\carrot B \back S\under k\carrot C, k \elementof N\subtwo\carrot 0}}\label{subsec:skq-cuts-combined}

In this section, we describe a polyhedral relaxation of the set $\PB \setminus \Sk^\Q$ for $k \in \Ncup$.

To begin, we consider the set $\PB \setminus \So^\Q$. The set $\So^\Q$ is equivalent to $\genset_{\D}^{\Q}$ from Section~\ref{subsec:tq-cuts} when $\D = \Nonetwo$. As such, the theory of Section~\ref{subsec:tq-cuts} can be applied to the specific case $\D = \Nonetwo$ to obtain an exponential family of inequalities for $\PB \setminus \So^\Q$ and a polynomial-size extended formulation of the polyhedron defined by these inequalities.

\addtocounter{example}{-1}
\begin{contexample}
	Let $\D$ from Section~\ref{subsec:tq-cuts} equal $\Nonetwo$. Consider $\P$ and $\Q$ defined in Example~\ref{ex:useless}. Figure~\ref{fig:pb-s0q-1} shows the selection of $\epsone_2(\Sset)$ for $\Sset = \{1\}$. This $\epsone_2(\Sset)$ is then used to construct the inequality of Theorem~\ref{thm:tq-cuts} in Figure~\ref{fig:pb-s0q-2}.
	
	\begin{figure}
			\begin{subfigure}[t]{0.44\textwidth}
				\vspace*{0pt}
				\centering
				\includegraphics[width=46mm]{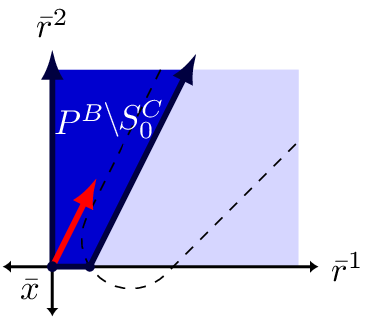}
				\caption{The maximal selection of $\epsone_2(\Sset)$, where $\Sset = \{1\}$. The vector $\enter_1 \rbar^1 + \epsone_2(\Sset) \rbar^2$ is depicted. If the weight on the term $\rbar^2$ were increased any further, the resulting vector would not lie in $\recc(\genset_{\D}^\Q)$.}
				\label{fig:pb-s0q-1}
			\end{subfigure}%
			\begin{subfigure}[t]{0.05\textwidth}
				\quad
			\end{subfigure}%
			\begin{subfigure}[t]{0.44\textwidth}
				\vspace*{0pt}
				\centering
				\includegraphics[width=46mm]{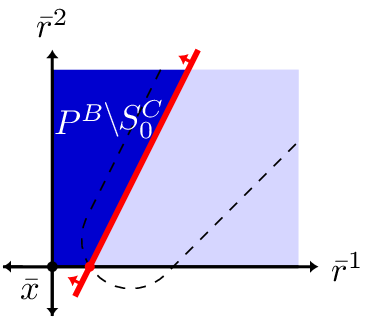}
				\caption{Theorem~\ref{thm:tq-cuts}'s valid inequality, $x_1 / \enter_1 - x_2 / \epsone_2(\Sset) \leq 1$. The corresponding hyperplane $\{x \in \R^2 \colon x_1 / \enter_1 - x_2 / \epsone_2(\Sset) = 1\}$ passes through the point $\xbasis + \enter_1 \rbar^1$. The vector $\enter_1 \rbar^1 + \epsone_2(\Sset) \rbar^2$ lies in the recession cone of this hyperplane.}
				\label{fig:pb-s0q-2}
			\end{subfigure}
		\caption{The valid inequality of Theorem~\ref{thm:tq-cuts} applied to Example~\ref{ex:useless}.}
	\end{figure}
\end{contexample}
\addtocounter{example}{1}

Now, let $k \in \N2$ be fixed. For the remainder of this section, we describe a polyhedral relaxation of $\PB \setminus \Sk^\Q$. Let $\J$ be defined as follows:
\begin{align*}
	\J &\coloneqq \{i \in N \colon \rbar^i \in \recc(\Sk^\Q) \}.
\end{align*}
Because $\recc(\Q) \subseteq \recc(\Sk^\Q)$, we have $\N1 \subseteq \J$.

\begin{observation}\label{obs:recc-skq}
	It holds that $\recc(\Sk^\Q) = \recc(\Q) + \rcinterval{-\infty}{0}{k}$.
\end{observation}

\begin{proposition}
	The index $k$ is not in $\J$.
\end{proposition}
\begin{proof}
	Assume for contradiction $k \in \J$. By Observation~\ref{obs:recc-skq}, there exists $q \in \recc(\Q)$ and $\lambda \geq 0$ such that $\rbar^k = q - \lambda \rbar^k$, which implies $\rbar^k \in \recc(\Q)$. This is a contradiction; $k \in \N2$, so the halfline $\lcinterval{0}{+\infty}{k}$ extending from $\xbasis$ intersects $\Q$ on a finite interval.
\end{proof}

Proposition~\ref{prop:skq-pb-int} characterizes the points where $\Sk^\Q$ intersects each edge of $\PB$.
\begin{proposition}\label{prop:skq-pb-int}
	Let $j \in N$. If Assumption~\ref{ass:recc} holds, then
	\begin{align*}
		\exitstar_j \coloneqq \sup \{\lambda \geq 0 \colon \xbasis + \lambda \rbar^j \in \Sk^\Q\} &=
		\begin{dcases}
			0 &\textrm{if } j \in \N0 \setminus \J \\
			\exit_j &\textrm{if } j \in \N2 \setminus \J \\
			+\infty &\textrm{if } j \in \J.
		\end{dcases}
	\end{align*}
\end{proposition}
\begin{proof}
	Let $j \in \J$. By Observation~\ref{obs:recc-skq}, there exists $\lambda \geq 0$ such that $\rbar^j + \lambda \rbar^k \in \recc(\Q)$. Consider any $\gamma > 0$. We have $\xbasis + \gamma(\rbar^j + \lambda \rbar^k) \in \Sk^\Q$. Because $-\rbar^k \in \recc(\Sk^\Q)$, we have $\xbasis + \gamma \rbar^j \in \Sk^\Q$. Thus, $\exitstar_j = +\infty$.
	
	Next, let $j \in \N2 \setminus \J$. By the construction of $\Sk^\Q$ in \eqref{eq:def-skq}, $\exitstar_j \geq \exit_j$. Assume for contradiction $\exitstar_j > \exit_j$. There exists $\theta \in \R^{|\N2|}_{+}$, $\delta \in \R^{|\N2|}_{+}$, $\gamma \geq 0$, and $q \in \recc(\Q)$ such that $\sum_{i \in \N2} \theta_i = 1$, $\delta_i \in [0, \exit_i)$ for $i \in \N2$, and
	\begin{align}
		\xbasis + \exitstar_j \rbar^j &= \xbasis + \sum\limits_{i \in \N2} \theta_i \delta_i \rbar^i - \gamma \rbar^k + q \nonumber \\
		\implies q &= \exitstar_j \rbar^j - \sum\limits_{i \in \N2} \theta_i \delta_i \rbar^i + \gamma \rbar^k. \label{eq:x-skq-cone}
	\end{align}
	Observe $\theta_i \delta_i = 0$ for all $i \in \N2 \setminus \{j,k\}$ and $\gamma \geq \theta_k \delta_k$; if not, $q \notin \recc(\PB)$ from \eqref{eq:x-skq-cone}, contradicting Assumption~\ref{ass:recc}. Therefore,
	\begin{align*}
		q &= (\exitstar_j - \theta_j \delta_j) \rbar^j + (\gamma - \theta_k \delta_k) \rbar^k.
	\end{align*}
	Because $\rbar^k \in \recc(\Sk^\Q)$, we have $q - (\gamma - \theta_k \delta_k) \rbar^k = (\exitstar_j - \theta_j \delta_j) \rbar^j \in \recc(\Sk^\Q)$. This contradicts $j \notin \J$.
	
	Finally, let $j \in \N0 \setminus \J$. Assume for contradiction $\exitstar_j > 0$. We follow the definitions in the previous case ($j \in \N2 \setminus \J$) to obtain
	\begin{align*}
		q &= \exitstar_j \rbar^j + (\gamma - \theta_k \delta_k) \rbar^k.
	\end{align*}
	Again, we obtain $\exitstar_j \rbar^j \in \recc(\Sk^\Q)$, contradicting $j \notin \J$.
\end{proof}

The proof of Proposition~\ref{prop:skq-pb-int} shows that without Assumption~\ref{ass:recc}, it may be the case that $\exitstar_j > \exit_j$ for some $j \in N$. This is due to the addition of $\rcinterval{-\infty}{0}{k}$ to $\recc(\Q)$.

Corollary~\ref{cor:n0-j} follows from Proposition~\ref{prop:skq-pb-int}.
\begin{corollary}\label{cor:n0-j}
	If $\N0 \subseteq \J$, then there exists $\epsilon > 0$ such that $\xbasis + \epsilon \rbar^j \in \Sk^\Q$ for all $j \in N$.
\end{corollary} 
By Corollary~\ref{cor:n0-j}, if $\N0 \subseteq \J$, $\xbasis$ lies in the relative interior of $\Sk^\Q$. We can construct a polyhedral relaxation of $\PB \setminus \Sk^\Q$ by using intersection cuts generated by the cone $\PB$. Methods for strengthening intersection cuts (e.g., \citet{glover1974}) can be used to obtain a strengthened polyhedral relaxation. For this reason, we present inequalities only for the case $\N0 \nsubseteq \J$.

\begin{assumption}\label{ass:2}
	There exists $j \in \N0$ such that $\rbar^j \notin \recc(\Sk^\Q)$, i.e., $\N0 \nsubseteq \J$.
\end{assumption}
For $i \in \J$ and $j \in N \setminus \J$, let
\begin{align*}
	\epshatthree_{ij} &\coloneqq \sup \{ \epshatthree \geq 0 \colon  \rbar^i + \epshatthree \rbar^j \in \recc(\Sk^\Q) \}.
\end{align*}
We define $\Msetthree$ to be the indices of $\J$ that satisfy the following property:
\begin{align*}
	\Msetthree &\coloneqq\{i \in \J \colon \epshatthree_{ij} > 0\ \forall j \in N \setminus \J \}.
\end{align*}
For any $i \in \Msetthree$ and $j \in N \setminus \J$, $\recc(\Sk^\Q)$ intersected with the cone $\F_{ij}$ contains something other than the trivial directions $\lcinterval{0}{+\infty}{i} \subseteq \recc(\Sk^\Q)$.

The proof of Proposition~\ref{prop:epshatthree-recc} is similar to that of Proposition~\ref{prop:epshat-recc}.
\begin{proposition}\label{prop:epshatthree-recc}
	Let $(i,j) \in \Msetthree \times (N \setminus \J)$. For any $\epshatthree \in [0,\epshatthree_{ij})$, we have $\rbar^i + \epshatthree \rbar^j \in \recc(\Sk^\Q)$.
\end{proposition}
For $\Sset \subseteq \Msetthree$ and $j \in N \setminus \J$, define $\epsthree_j(\Sset)$ to be
\begin{align*}
	\epsthree_j(\Sset) &=
	\begin{cases}
		\min_{i \in \Sset} \epshatthree_{ij} & \textrm{ if } \Sset \neq \emptyset \\
		+\infty & \textrm{ otherwise}.
	\end{cases}
\end{align*}
By Proposition~\ref{prop:epshatthree-recc}, if $\Sset \neq \emptyset$, $\rbar^i + \epsthree_j(\Sset) \rbar^j \in \recc(\Sk^\Q)$ for all pairs $(i,j) \in \Sset \times (N \setminus \J)$.
\begin{theorem}\label{thm:skq-cuts}
	Let $\Sset \subseteq \Msetthree$. The inequality
	\begin{align}
		\sum\limits_{i \in U} x_i - \mash{\sum\limits_{j \in N \setminus \J}} \ \frac{x_j}{\epsthree_j(\Sset)} &\leq 0 \label{eq:skq-inequality}
	\end{align}
	is valid for $\PB \setminus \Sk^\Q$.
\end{theorem}
\begin{proof}
	Assume $\Sset \neq \emptyset$, or the result trivially holds. By construction, $\epsthree_j(\Sset) > 0$. Let $\xhat \in \PB$ satisfy $\sum_{i \in U} \xhat_i- \sum_{j \in N \setminus \J} \xhat_j / \epsthree_j(\Sset) > 0$. We show $\xhat \in \Sk^\Q$. For ease of notation, let $\epsthree_j \coloneqq \epsthree_j(\Sset)$.
	
	By Proposition~\ref{prop:epshatthree-recc}, for $(i,j) \in \Sset \times (N \setminus \J)$, there exists $q^{ij} \in \recc(\Sk^\Q)$ such that $q^{ij} = \rbar^i + \epsthree_j \rbar^j$. Then
	\begin{alignat*}{2}
		\rbar^j &= \frac{1}{\epsthree_j} q^{ij} - \frac{1}{\epsthree_j} \rbar^i \qquad &&\forall i \in \Sset,\ j \in N \setminus \J.
	\end{alignat*}
	By Lemma~\ref{lem:farkas-cons-2}, there exists $\theta \in \R_{+}^{|\Sset| \times |N \setminus \J|}$ such that
	\begin{subequations}\label{eq:farkas-in-skq}
		\begin{alignat}{2}
			\sum\limits_{i \in \Sset} \theta_{ij} &= 1 \qquad && \forall j \in N \setminus \J \label{eq:farkas-in-skq-1} \\
			\mash{\sum\limits_{j \in N \setminus \J}} \ \theta_{ij} \frac{\xhat_j}{\epsthree_j} &\leq \xhat_i && \forall i \in \Sset. \label{eq:farkas-in-skq-2}
		\end{alignat}
	\end{subequations}
	This result is obtained with $\indone \coloneqq \Sset$, $\indtwo \coloneqq N \setminus \J$, $\aparam_i \coloneqq \xhat_i$ for all $i \in \Sset$, and $\cparam_j \coloneqq \xhat_j / \epsthree_j$ for all $j \in N \setminus \J$. With the $\theta$ satisfying \eqref{eq:farkas-in-skq}, we have
	\begin{align}
		\rbar^j &= \sum\limits_{i \in \Sset} \theta_{ij} \bigg( \frac{1}{\epsthree_j} q^{ij} - \frac{1}{\epsthree_j} \rbar^i \bigg) \qquad \forall j \in N \setminus \J. \label{eq:new-rbarj}
	\end{align}
	Using \eqref{eq:new-rbarj}, $\xhat$ is equivalent to
	\begin{align*}
		\xhat &= \xbasis + \sum\limits_{i \in \Sset} \xhat_i \rbar^i + \mash{\sum\limits_{i \in \J \setminus \Sset}} \xhat_i \rbar^i
		+ \mash{\sum\limits_{j \in N \setminus \J}} \xhat_j \rbar^j \\
		&= \xbasis + \sum\limits_{i \in \Sset} \xhat_i \rbar^i + \mash{\sum\limits_{j \in N \setminus \J}} \xhat_j \sum\limits_{i \in \Sset} \theta_{ij} \bigg( \frac{1}{\epsthree_j} q^{ij} - \frac{1}{\epsthree_j} \rbar^i \bigg) + \mash{\sum\limits_{i \in \J \setminus \Sset}} \xhat_i \rbar^i \\
		&= \xbasis + \sum\limits_{i \in \Sset} \bigg( \xhat_i - \mash{\sum\limits_{j \in N \setminus \J}} \theta_{ij} \frac{\xhat_j}{\epsthree_j} \bigg) \rbar^i + \mash{\sum\limits_{j \in N \setminus \J}} \enspace \sum\limits_{i \in \Sset} \theta_{ij} \frac{\xhat_j}{\epsthree_j} q^{ij} + \mash{\sum\limits_{i \in \J \setminus \Sset}} \xhat_i \rbar^i.
	\end{align*}
	By \eqref{eq:farkas-in-skq-2}, the coefficients on the terms $\rbar^i$, $i \in \Sset$ are nonnegative. Observe that
	\begin{alignat*}{2}
		\bigg( \xhat_i - \mash{\sum\limits_{j \in N \setminus \J}} \theta_{ij} \frac{\xhat_j}{\epsthree_j} \bigg) \rbar^i &\in \recc(\Sk^\Q) \qquad && \forall i \in \Sset \\
		\theta_{ij} \frac{\xhat_j}{\epsthree_j} q^{ij} &\in \recc(\Sk^\Q) && \forall i \in \Sset,\ j \in N \setminus \J \\
		\xhat_i \rbar^i &\in \recc(\Sk^\Q) && \forall i \in \J \setminus \Sset.
	\end{alignat*}
	It follows that $\xhat \in \{\xbasis\} + \recc(\Sk^\Q) \subseteq \Sk^\Q$.
\end{proof}

We next consider the separation problem for $\skrelax \supseteq \PB \setminus \Sk^\Q$, where
\begin{align*}
	\skrelax &\coloneqq \bigg\lbrace x \in \R^{|N|}_{+} \colon \sum\limits_{i \in U} x_i - \mash{\sum\limits_{j \in N \setminus \J}} \ \frac{x_j}{\epsthree_j(\Sset)} \leq 0 \ \forall \Sset \subseteq \Msetthree \bigg\rbrace.
\end{align*}
In particular, given some $\xhat \in \PB$, we are interested in finding a subset of $\Msetthree$ that maximizes the violation of an inequality of the form \eqref{eq:skq-inequality}:
\begin{align}
	\max_{\Sset \subseteq \Msetthree}\sum\limits_{i \in U} \xhat_i - \mash{\sum\limits_{j \in N \setminus \J}} \ \frac{\xhat_j}{\epsthree_j(\Sset)}. \label{eq:skq-separation}
\end{align}
\begin{proposition}\label{prop:skq-supermodular}
	The separation problem \eqref{eq:skq-separation} is a supermodular maximization problem.
\end{proposition}

Similar to the derivation of $\tqextended_{\D}$ in Section~\ref{subsec:tq-cuts}, we derive an extended formulation for the relaxation of $\PB \setminus \Sk^\Q$ defined by inequality \eqref{eq:skq-inequality} for all $\Sset \subseteq \Msetthree$. Let $\Msetthree = \{1, \ldots, \skcardone\}$, where $\skcardone \coloneqq |\Msetthree|$. Let $\skcardtwo \coloneqq |N \setminus \J|$. For all $j \in N \setminus \J$, let $\pi_j(1), \pi_j(2), \ldots, \pi_j(\skcardone)$ be ordered to satisfy $\epshatthree_{\pi_j(1),j} \leq \epshatthree_{\pi_j(2),j} \leq \ldots \leq \epshatthree_{\pi_j(\skcardone),j}$. For $i \in \Msetthree$, let $\ellone_j(i)$ be the unique integer satisfying $\pi_j(\ellone_j(i)) = i$. For all $j \in N \setminus \J$, let $\epshatthree_{0j} \coloneqq +\infty$, $\theta_{0j} \coloneqq 0$, $v_{0j} \coloneqq 0$, $v_{\skcardone+1,j} \coloneqq 0$, $\pi_j(0) \coloneqq 0$, and  $\pi_j(\skcardone+1) \coloneqq 0$. We define $\skextended$ to be the set of $(x,\theta,v,\lambda) \in \R_{+}^{|N|} \times \R_{+}^{\skcardone \times \skcardtwo} \times \R_{+}^{\skcardone \times \skcardtwo} \times \R^{\skcardtwo}$ such that
\begin{equation*}
	\begin{alignedat}{2}
		\sum\limits_{i \in \Msetthree} \enspace\ \mash{\sum\limits_{j \in N \setminus \J}} \theta_{ij} + \mash{\sum\limits_{j \in N \setminus \J}} \lambda_j &\leq 0 \\
		\theta_{ij} + v_{ij} - v_{i+1,j} + \bigg( \frac{1}{\epshatthree_{\pi_j(i+1),j}} - \frac{1}{\epshatthree_{\pi_j(i),j}} \bigg) x_j &\geq 0 \quad && \forall i = 0, \ldots, \skcardone,\ j \in N \setminus \J \\
		\mash{\sum\limits_{j \in N \setminus \J}} \theta_{\ellone_j(i),j} - x_i &\geq 0 \quad && \forall i = 1, \ldots, \skcardone.
	\end{alignedat}
\end{equation*}

\begin{theorem}\label{thm:extendedform3}
	It holds that $\proj_x(\skextended) = \skrelax$.
\end{theorem}
The proof of Theorem~\ref{thm:extendedform3} is left out, because it mirrors that of Theorem~\ref{thm:extendedform1}. We can use the extended formulation $\proj_x(\skextended)$ to construct a polyhedral relaxation of $\PB \setminus \Sk^\Q$ from the multi-term disjunction \eqref{eq:disjunction}.

The nontrivial inequalities of Theorem~\ref{thm:skq-cuts} are predicated on the existence of a nonempty $\Sset \subseteq \Msetthree$. We end this section by stating that if no such subset exists (i.e., $\Msetthree = \emptyset$), then no nontrivial inequalities exist for $\PB \setminus \Sk^\Q$.
\begin{proposition}\label{prop:msetthree-sufficient}
	Under Assumption~\ref{ass:2}, if $\Msetthree = \emptyset$, then $\clconv(\PB \setminus \Sk^\Q) = \PB$.
\end{proposition}
\begin{proof}
	By Assumption~\ref{ass:2}, $\N0 \setminus \J \neq \emptyset$. We show $\{\xbasis\} + \lcinterval{0}{+\infty}{i} \subseteq \clconv(\PB \setminus \Sk^\Q)$ for all $i \in N$. Observe $\xbasis \in \cl(\PB \setminus \Sk^\Q)$ by Proposition~\ref{prop:skq-pb-int}.
	
	Consider any $i \in N \setminus \J$ and $\gamma > 0$. We show $\xbasis + \gamma \rbar^i \in \clconv(\PB \setminus \Sk^\Q)$. By Proposition~\ref{prop:skq-pb-int}, $\exitstar_i$ is finite. Then for a sufficiently large $M > \gamma$, $\xbasis + M \rbar^i \notin \Sk^\Q$. We have that $\xbasis + \gamma \rbar^i$ is a convex combination of $\xbasis \in \cl(\PB \setminus \Sk^\Q)$ and $\xbasis + M \rbar^i \in \PB \setminus \Sk^\Q$. Hence, $\xbasis + \gamma \rbar^i \in \clconv(\PB \setminus \Sk^\Q)$.
	
	Now, consider $i \in \J$, $\lambda > 0$, and $\epshatthree > 0$. Because $\Msetthree = \emptyset$, there exists $j \in N \setminus \J$ such that $\lambda \rbar^i + \epshatthree \rbar^j \notin \recc(\Sk^\Q)$. Then there exists $M > 1$ such $\xbasis + M(\lambda \rbar^i + \epshatthree \rbar^j)$ lies outside of $\Sk^\Q$. Therefore, $\xbasis + \lambda \rbar^i + \epshatthree \rbar^j$ is a convex combination of $\xbasis \in \cl(\PB \setminus \Sk^\Q)$ and $\xbasis + M(\lambda \rbar^i + \epshatthree \rbar^j) \in \PB \setminus \Sk^\Q$. This holds for an arbitrary $\epshatthree > 0$, so $\xbasis + \lambda \rbar^i \in \clconv(\PB \setminus \Sk^\Q)$.
\end{proof}

\section{Discussion and future work}

Our analysis requires the basic solution $\xbasis$ to lie outside $\cl(\Q)$. We showed in Section~\ref{sec:cuts-1} that if $\xbasis \in \Q$, we obtain the standard intersection cut of Balas. It remains to discuss how we can derive valid inequalities for $\P \setminus \Q$ when $\xbasis \in \bd(\Q)$.

Under Assumption~\ref{ass:1}, our analysis still applies if $\xbasis \in \bd(\Q)$. To demonstrate this, assume for simplification that $\N0 = \emptyset$ (this is a more restrictive version of Assumption~\ref{ass:1}). It follows that $\enter_j = 0$ for all $j \in N$. Similar to the observation made in Remark~\ref{remark:1} for the case $\xbasis \in \Q$, we can show that every point in $\PB \setminus \Q$ lies in $\{\xbasis\}$ or $\{x \in \PB \colon \sum_{j \in N} x_j / \exit_j \geq 1 \}$. We can generate inequalities for $\P \setminus \Q$ in a disjunctive CGLP using the two polyhedra defined by the constraints of $\P$ added to each of these two sets. Similarly, if $\xbasis \in \bd(\Q)$ and Assumption~\ref{ass:1} holds, the term $\PB \setminus \So^\Q$ of the multi-term disjunction \eqref{eq:disjunction} is equal to $\{\xbasis\}$. We can again use disjunctive programming to generate cuts for $\P \setminus \Q$ with the knowledge that $\PB \setminus \So^\Q = \{\xbasis\}$. Polyhedral relaxations for the remaining disjunctive terms can still be generated using the methods discussed in Section~\ref{subsec:skq-cuts-combined}.

We conclude with some ideas for future work. One direction is to study the computational strength of cuts obtained using these ideas. Another possibility is to generalize this disjunctive framework to allow for cuts to be generated by bases of rank less than $m$ (i.e., bases that do not admit a basic solution). Additionally, the strength of $\tqrelax_{\D}$ relative to $\PB \setminus \genset_{\D}^\Q$ could be analyzed. Specifically, it remains to be seen if $\tqrelax_{\D} = \conv(\PB \setminus \genset_{\D}^\Q)$, which by Theorem~\ref{thm:extendedform1} would imply that we have a polynomial-size extended formulation of $\conv(\PB \setminus \So^\Q)$. The same applies to the strength of $\skrelax$ relative to $\PB \setminus \Sk^\Q$ for $k \in \N2$.

\bibliographystyle{informs2014} 
\bibliography{refs} 

\end{document}